\documentclass[11pt,reqno]{amsart}
\usepackage{amssymb, amsmath, amsthm}
\usepackage{hyperref}
\usepackage{cleveref}
\usepackage{tikz-cd}
\usepackage{mathtools}
\usepackage[ruled]{algorithm2e}
\usepackage{enumerate}

 \usepackage{hyperref}

\usepackage[alphabetic,lite]{amsrefs}

\newtheorem{theorem}{Theorem}[section]
\newtheorem{prop}[theorem]{Proposition}

\newtheorem{lemma}[theorem]{Lemma}
\newtheorem{cor}[theorem]{Corollary}
\newtheorem{prob}[theorem]{Problem}
\newtheorem{question}[theorem]{Question}

\theoremstyle{definition}
\newtheorem{defn}[theorem]{Definition}

\newtheorem{ex}[theorem]{Example}
\newtheorem{remark}[theorem]{Remark}
\numberwithin{equation}{section}

\newtheorem*{class}{Theorem on Classification of Invariant Ideals}

\DeclareMathOperator{\GL}{GL}
\DeclareMathOperator{\Char}{char}
\DeclareMathOperator{\lcm}{lcm}

\DeclareMathOperator{\pd}{pd}
\DeclareMathOperator{\depth}{depth}

\renewcommand{\k}{\mathbf{k}}
\newcommand{\Sym}{\operatorname{Sym}}
\newcommand{\defi}[1]{{\upshape\sffamily #1}}
\newcommand{\fl}[1]{#1^{\circ}}
\newcommand{\m}{\mathfrak{m}}

\newcommand{\Cont}[1]{{\left \lVert #1 \right \rVert}}
\newcommand{\cont}[1]{{\lVert #1 \rVert}}

\title{Ideals Preserved By Linear Changes of Coordinates in Positive Characteristic}

\author[Cattell-Ravdal]{Bj{\o}rn Cattell-Ravdal}
\address{Metropolitan State University of Denver}
\email{bcattell@msudenver.edu}

\author[Delargy]{Erin Delargy}
\address{Binghamton University}
\email{edelarg1@binghamton.edu}

\author[Ganguly]{Akash Ganguly}
\address{Carleton College}
\email{gangulya@carleton.edu}

\author[Guan]{Sean Guan}
\address{University of Califonia, Berkeley}
\email{seanguan@berkely.edu}

\author[Karn]{Trevor Karn}
\address{University of Minnesota - Twin Cities}
\email{karnx018@umn.edu}

\author[Perlman]{Michael Perlman}
\address{University of Minnesota - Twin Cities}
\email{mperlman@umn.edu}

\author[Sivakumar]{Saisudharshan Sivakumar}
\address{University of Florida}
\email{sivakumars@ufl.edu}

\begin{document}
\maketitle 
\begin{abstract}
We consider the polynomial ring in finitely many variables over an algebraically closed field of positive characteristic, and initiate the systematic study of ideals preserved by the action of the general linear group by changes of coordinates. We show that these ideals are classified by sets of carry patterns, which are finite sequences of integers introduced by Doty in the study of representation theory of the polynomial ring. We provide an algorithm to decompose an invariant ideal as a sum of carry ideals with no redundancies. Next, we study the conditions under which one carry ideal is contained in another, and completely characterize the image of the multiplication map between the space of linear forms and a subrepresentation of forms of degree $d$. Finally, we begin an investigation into free resolutions of these ideals. Our results are most explicit in the case of carry ideals in two variables, where we completely describe the monomial generators and syzygies using base-$p$ expansions of the parameters involved, and we provide a formula for the structure of the Tor modules in the Grothendieck group of representations.
\end{abstract}
\setcounter{tocdepth}{1}

\section{Introduction}
Let $V$ be an $n$-dimensional vector space over an algebraically closed field $\k$ of characteristic $p>0$. We consider the polynomial ring $S=\Sym(V)$ endowed with the natural action of the general linear group $\GL(V)$ via changes of coordinates. If we choose a basis $V=\k\langle x_1,\cdots,x_n\rangle$, then we have identifications $S=\k[x_1,\cdots,x_n]$ and $\GL(V)=\GL_n(\k)$. The action of $g=(g_{i,j})\in \GL_n(\k)$ on a monomial $x_1^{b_1}x_2^{b_2}\cdots x_n^{b_n}$ is 
$$
g\cdot (x_1^{b_1}x_2^{b_2}\cdots x_n^{b_n})=(g_{1,1}x_1+\cdots +g_{n,1}x_n)^{b_1}(g_{1,2}x_1+\cdots +g_{n,2}x_n)^{b_2}\cdots (g_{1,n}x_1+\cdots +g_{n,n}x_n)^{b_n}.
$$
In other words, $g$ sends $x_i$ to the dot product of the $i$-th column of $g$ with the vector $[x_1\;x_2\;\cdots\;x_n]$.

In this paper we initiate the systematic study of  \defi{invariant ideals} of $S$ with respect to this action, i.e. those ideals $I\subseteq S$ for which $g\cdot I \subseteq I$ for all $g\in \GL_n(\k)$. Such ideals are necessarily $\mathfrak{S}_n$-invariant monomial ideals, as they are invariant under the subgroup of permutation matrices in $\GL_n(\k)$, as well as the maximal torus of diagonal matrices. The property of being invariant is dependent upon characteristic. In characteristic zero, the only invariant ideals are powers of the homogeneous maximal ideal, due to the fact that the graded pieces $S_d=\Sym^d(\mathbb{C}^n)$ are irreducible representations of $\GL_n(\mathbb{C})$. On the other hand, the ideal $I=\langle x^2, y^2\rangle\subseteq \k[x,y]$ is $\GL_n(\k)$-invariant if and only if $p=2$ (since $2xy=0$). More generally, if $q=p^e$ for some $e\geq 0$, then $I=\langle x_1^q,x_2^q,\cdots,x_n^q\rangle$ is $\GL_n(\k)$-invariant due to the existence of the Frobenius endomorphism. Not all invariant ideals are of this form. For example, $I=\langle x^4, x^3y, xy^3, y^4\rangle\subseteq \k[x,y]$ is $\GL_2(\k)$-invariant if and only if $p$ is $2$ or $3$.

The first main problem we investigate in this work is:

\begin{prob}
Classify all $\GL_n(\k)$-invariant ideals in $S=\k[x_1,\cdots,x_n]$, and parametrize them with combinatorial and/or arithmetic data that encodes their algebraic structure.
\end{prob}

Our goal is to develop a theory analogous to that of $\GL_m(\mathbb{C})\times \GL_n(\mathbb{C})$-invariant ideals in the polynomial ring $\Sym(\mathbb{C}^m\otimes \mathbb{C}^n)\cong \mathbb{C}[x_{i,j}]$, known as determinantal thickenings \cite{DEP}. There, invariant ideals are parametrized by collections of Young diagrams, and many algebraic constructions are encoded via combinatorics of those diagrams. For instance, one may use the Young diagrams indexing an invariant ideal to test ideal containment, find radicals, and to calculate primary decompositions. 

In our setting, invariant ideals are parametrized by sets of \defi{carry patterns}, which are finite sequences of integers associated to base-$p$ expansions of the exponents of the minimal generators of the ideal. We say that a monomial $x_1^{b_1}x_2^{b_2}\cdots x_n^{b_n}$ has carry pattern $c=(c_1,\cdots,c_M)$ if $c_i$ is the amount carried to the $p^i$-column when adding $b_1,b_2,\cdots,b_n$ in base-$p$ (see Section \ref{sec:carrypatterns} for precise definition). These sequences appeared in the work of Doty \cite{Doty} on the structure of the polynomial ring $S$ as a representation of $\GL_n(\k)$. 

Given an integer $d\geq 0$ and a carry pattern $c$, we define the \defi{carry ideal} $I_{c,d}$ to be the ideal generated by monomials of degree $d$ with carry pattern $\leq c$, where the partial order on carry patterns is defined by $c'\leq c$ if $c_i'\leq c_i$ for all $i\geq 1$. For instance, when $p=2$ the ideal $I=\langle x^4, x^3y, xy^3, y^4\rangle\subseteq \k[x,y]$ from above is the carry ideal $I_{(1,0),4}$, as $x^4,y^4$ have carry pattern $(0,0)$ and $x^3y,xy^3$ have carry pattern $(1,0)$, and the missing monomial $x^2y^2$ has carry pattern $(0,1)$. 

Our first main result shows that carry ideals are the building blocks for all invariant ideals.

\begin{class}[Theorem \ref{structure theorem}]
Let $\k$ be an algebraically closed field of positive characteristic. An ideal in $\k[x_1,\cdots,x_n]$ is $\GL_n(\k)$-invariant if and only if it is a sum of carry ideals. 
\end{class}

We prove this result in Section \ref{sec:class} using Doty's work on the $\GL_n(\k)$-submodule lattice of the graded pieces $S_d$ of $S=\k[x_1,\cdots,x_n]$, reviewed in Section \ref{sec:rep}. In Section \ref{sec:class} we provide an algorithm to decompose an invariant ideal into a sum of carry ideals with no redundancies, and in Section \ref{sec:powers} we investigate the effect of powers and Frobenius powers on invariant ideals. In Section \ref{sec:contain} we study the conditions under which one carry ideal is contained in another. Towards this end, Theorem \ref{carry multiply x_i induction} describes how the carry pattern of a monomial $x^{\underline{b}}$ determines the carry pattern of the monomial $x_i x^{\underline{b}}$ obtained after multiplication by a variable. As an application, we completely describe the image of the multiplication map $S_1\otimes W\to S_{d+1}$ when $W$ is a subrepresentation of $S_d$ (see Theorem \ref{thm:multmap}), thereby giving an inductive method to characterize containment relations between two carry ideals.

The second main problem we investigate is:

\begin{prob}
Given an invariant ideal $I\subseteq S$, calculate its minimal free resolution and determine the structure of each $\operatorname{Tor}_i^S(S/I,\k)$ as a representation of $\GL_n(\k)$.
\end{prob}

Since $I$ is an invariant ideal, it follows that $S/I$ is supported on the origin, so  $\pd_S(S/I)=n$  (see Proposition \ref{prop:pd}), and the Castelnuovo--Mumford regularity of $S/I$ is the maximal degree $d$ for which $(S/I)_d\neq 0$ (see Proposition \ref{regprop}). We describe a general procedure to get information about the $\GL_n(\k)$-structure of $\operatorname{Tor}_i^S(S/I,\k)$ using the Koszul complex (see Section \ref{sec:reptor}). Our results are most explicit in the case of two variables (see Section \ref{gens2} and Section \ref{section: Two variable case}), summarized as follows:
\begin{enumerate}   
    \item Not every finite sequence of non-negative integers appears as a carry pattern, and Doty describes the set $C(d,n,p)$ of carry patterns of degree $d$ monomials in $n$ variables in characteristic $p$ in a non-constructive fashion, via a system of equations (see Lemma \ref{lemma: two conditions on c}). In Proposition \ref{prop: carry only 0 or 1} we provide a complete list of elements in $C(d,2,p)$. In particular, every such carry pattern has entries in $\{0,1\}$.

    \item The work of Doty suggests a naive algorithm to find the generators of $I_{c,d}$: list all monomials of degree $d$, calculate their carry patterns, and eliminate those that do not have carry pattern $\leq c$. In two variables, we provide a closed formula for the generators of  $I_{c,d}$ by showing that it is a product of Frobenius powers of powers of the homogeneous maximal ideal (see Theorem \ref{thm:genstwovars}). The powers are determined by what we call the \textit{type $c$ segmentation of $d$} (see Section \ref{gens2}).

    \item We show that the quotient of $S$ by a carry ideal $I_{c,d}$ in two variables has a Hilbert--Burch resolution, and we provide closed formulas for the graded Betti numbers and Castelnuovo--Mumford regularity in terms of the type $c$ segmentation of $d$. As an application, we provide a formula for the spaces $\operatorname{Tor}_i^S(S/I_{c,d},\mathbf{k})_j$ in the Grothendieck group of representations of $\GL_2(\k)$.
\end{enumerate}

We end Section \ref{freeres} with a brief discussion of \textit{equivariant} free resolutions of invariant ideals and an example.

\section{Carry patterns and representation theory}\label{section: Carry Patterns}

Throughout this paper, $\k$ is an algebraically closed field of characteristic $p>0$, and $S= \k[x_1,\dots,x_n]$. In this section we recall information about the $\GL_n(\k)$-representation theory of the polynomial ring $S$.

\subsection{Base-$p$ expansions and carry patterns}\label{sec:carrypatterns} The base-$p$ expansion of an integer $d\geq 0$ is written
$$
d=(d_0,d_1,\cdots, d_M),
$$
where $0\leq d_0,\cdots,d_M\leq p-1$ satisfy
\[d = \sum_{j=1}^M d_j p^j.\]
We often write $M$ for the maximal index such that $d_M\neq 0$. We define 
\begin{equation}\label{eq:floor}
 \fl{d}=\min\{j|d_j\neq p-1\}.   
\end{equation}

Given a monomial $x^{\underline{b}}=x_1^{  b_1}x_2^ {b_2}\cdots x_n^{ b_n}\in S$ of degree $d = b_1+\cdots + b_n$, we frequently consider the base-$p$ expansions of the $b_i$. We will write $b_{i,j}$ for the coefficients of $p^j$ in the base-$p$ of $b_i$. Note that if $M$ is the largest $j$ such that $d_j\neq 0$, then for $j>M$, every $ b_{i,j}=0$. On the other hand, for $j< M$ it is possible that $b_{i,j}\geq d_j$ for some $i$.
This can happen when we need to ``carry" values from the $p^j$ place to the $p^{j+1}$ place during the process of addition with regrouping. This carrying is the usual notion of carrying we are accustomed to when adding numbers base-$10$.
Following Doty \cite{Doty} we define sequences of integers that encode the carries that occur. Let $x^{\underline{b}}\in S$ be a monomial of degree $d$.
    The \defi{carry pattern} $c(\underline{b})$ (or $c(x^{\underline{b}})$) of  $x^{\underline{b}}$ is the tuple of integers $(c_1(\underline{b}),c_2(\underline{b}),...,c_M(\underline{b}))$ defined by the following system of equations:
    \begin{equation}\label{carry pattern}
 \sum_{i=1}^n\sum_{0\leq j<\ell} b_{ij}p^j=c_\ell(\underline{b})p^\ell+\sum_{0\leq j<\ell}d_jp^j, \quad 1\leq \ell \leq M,
\end{equation}
    with the convention that $c_i(\underline{b})=0$ for $i<1$ and $i>M$. A more intuitive way to understand carry patterns is to view each $c_j(\underline{b})$ as the amount ``carried to the $p^j$ column" when adding the entries of $\underline{b}$ in base-$p$.

\begin{ex}
Let $S=\k[x,y]$ where $\Char(\k)=3$. Consider the monomial $x^{4}y^{6}$ of degree $d=10$. Then $d$ has the base-$3$ expression $d = (1,0,1)$, so that $M=2$. 
The base-$3$ expansions 
of the exponents are 
\[  b_1 =4=(1,1),\quad  b_2 = 6 = (0,2).\]

We add the base-$3$ expansions of the exponents, keeping track of how much is carried to the next column:
    \begin{center}
        \begin{tabular}{c c}
             $b_{1,1}$& $b_{1,0}$ \\
             $b_{2,1}$& $b_{2,0}$ 
        \end{tabular}
        \hspace{0.15in}$\longrightarrow$\hspace{0.15in}
        \begin{tabular}{c c c c }
         &$\mathbf{1}$ & $\mathbf{0}$ & \\
         && $1$ & $1$ \\
         +&& $2$ & $0$ \\
         \hline
         &1 & 0 & 1
    \end{tabular}
    \end{center}
    So $c(4,6) = (c_1(4,6), c_2(4,6)) = (0, 1)$. 

We could compute this more formally using (\ref{carry pattern}). Then $c_1(4,6)$ and $c_2(4,6)$ are given by 
\[1=c_1(4,6)\cdot 3+1\quad \text{and}\quad 1+1\cdot 3+2\cdot 3 =c_2(4,6)\cdot 3^2+1.\]
Thus $c_1(4,6)=0$, and $c_2(4,6)=1$, so again we see $c(4,6) = (0,1)$. 
\end{ex}

The next example shows that $c_j(\underline{b})$ can be greater than $p-1$.

\begin{ex}   
    Let $S=\k[x,y,z]$ with $\Char(\k)=2$, and consider the monomial $x^3y^3z^3$ with degree $d=9=(1,0,0,1)$, so $M=3$. We add the base-$2$ expansions of each of the exponents, keeping track of how much is carried to the next column:
    \begin{center}
    \begin{tabular}{c c c c c }
         &$\mathbf{1}$& $\mathbf{2}$ & $\mathbf{1}$ & \\
         &&$0$& $1$ & $1$ \\
         &&$0$& $1$ & $1$ \\ 
         +&&$0$& $1$ & $1$ \\
         \hline
         & 1 & $0$ & $0$ & $1$
    \end{tabular}
    \end{center}
    Thus, $c(\underline{b})=(1,2,1)$. 
\end{ex}

Given integers $n\geq 1$ and $d\geq 0$, we write $C(d,n,p)$ for the set of all carry patterns of degree $d$ monomials in $\mathbf{k}[x_1,\cdots, x_n]$. The following result describes $C(d,n,p)$ via a system of equations.

\begin{lemma}[\cite{Doty} Lemma 3] \label{lemma: two conditions on c}
    Let $(c_1,c_2,\cdots ,c_M)$ be an $M$-tuple of integers. Then $c \in C(d,n,p)$ if and only if the following inequalities hold for all $i\geq 1$:
    $$
    0\le c_i \le \sum_{j\ge i}d_j p^{j-i}\quad \textnormal{and}\quad 0\le d_i +p c_{i+1}-c_i\le n(p-1).
    $$
\end{lemma}

For any pair of sequences $c = (c_1, c_2, c_3, \ldots) $
and $c' = (c'_1, c'_2, c'_3,\ldots)$, we define the relation $\le$ given by component-wise comparison: $c \le c'$ if $c_i \le c_i'$ for all $i$.
We endow $C(d,n,p)$ with this partial order. 

\begin{cor}[\cite{Doty} Corollary 4]\label{mincarry}
    There exists a unique minimal carry pattern in $C(d,n,p)$, namely the carry pattern $(0,\cdots,0)$. 
\end{cor}

We extend this result to show that $C(d,n,p)$ has a unique maximal element for \emph{any} $(d,n,p)$. Before we state the result, we define two operations on carry patterns. Given two carry patterns $c$ and $c'$, we define \[\lcm(c,c')=(\max(c_1,c_1'),\max(c_2,c_2'),...,\max(c_M,c_M')).\] 
    Similarly, we define
    \[\gcd(c,c') = (\min(c_1, c_1'), \ldots, \min(c_M, c_M')).\]

\begin{prop}\label{prop:lcm is join}
    If $c,c' \in C(d,n,p)$ then both $\lcm(c,c')$ and $\gcd(c,c')$ belong to $C(d,n,p)$. The poset $C(d,n,p)$ is a finite lattice, and thus has a maximum (a unique maximal element).
\end{prop}

\begin{proof}
We only prove the assertion about $\lcm(c,c')$, as the proof that $\gcd(c,c')\in C(d,n,p)$ is similar. We  consider two carry patterns $c=(c_1,\cdots,c_M)$ and $c'=(c_1',\cdots,c_M')$ in $C(d,n,p)$. We will show that $\lcm(c,c')$ satisfies the two conditions given by Lemma \ref{lemma: two conditions on c}. By Lemma \ref{lemma: two conditions on c} we know that $0\le c_i \le \sum_{j\ge i}d_j p^{j-i}$ holds for each $c_i$; the same inequality holds for each $c_i'$. Hence $0\le \max(c_i,c_i') \le \sum_{j\ge i}d_j p^{j-i}$, and so $\lcm(c,c')$ satisfies the first condition in Lemma \ref{lemma: two conditions on c}.

Now we consider the second condition of Lemma \ref{lemma: two conditions on c}. Again, by Lemma \ref{lemma: two conditions on c}, we have
\[
0\le d_i +p c_{i+1}-c_i\le n(p-1),
\quad \text{and} \quad
0\le d_i +p c'_{i+1}-c'_i\le n(p-1),
\] for all $i$.
If $\max(c_i,c'_i) = c_i$ and $\max(c_{i+1},c'_{i+1}) = c_{i+1}$ then we are done. Similarly if $\max(c_i,c'_i) = c'_i$ and $\max(c_{i+1},c'_{i+1}) = c'_{i+1}$ then we are done. Without loss of generality, consider the case where $\max(c_i,c'_i) = c_i$ and $\max(c_{i+1},c'_{i+1}) = c'_{i+1}$. But since $c'_i \leq c_i$ and $c_{i+1} \leq c'_{i+1}$, we have
\[
d_i +pc_{i+1}-c_i \leq d_i +pc'_{i+1}-c_i\leq d_i +p c'_{i+1}-c'_i,
\]
which implies
\[
0\leq d_i +p c'_{i+1}-c_i\leq n(p-1),
\]
as desired. 

It is straightforward to see that $\lcm(c,c')$ is the \emph{unique} minimal element which is greater than both $c, c'$ and similarly that $\gcd(c, c')$ is the unique minimal element which is less than both $c, c'$. This implies that $C(d,n,p)$ is a lattice. Since $C(d,n,p)$ is finite, this implies that it has a maximum element. 
\end{proof}

\begin{ex}\label{degree 10 cp example}
    Let $S=\k[x,y]$ with $\Char(\k)=2$. Consider the set of all carry patterns of degree ten monomials:
    \[C(10,2,2)=\{(0,0,0), (1,0,0), (0,0,1), (1,0,1), (1,1,1)\}\]

The following table describes the carry patterns of the monomials in $S_{10}$.
    \vspace{0.1in}
    \begin{center}
    \begin{tabular}{c|c}
        $c(x^{\underline{b}})$ & $x^{\underline{b}}$ \\ \hline
        $(1,1,1)$& $x^7y^3,x^3y^7$\\
        $(1,0,1)$& $x^5y^5$\\
        $(1,0,0)$&$x^9y,xy^9$\\
        $(0,0,1)$&$x^6y^4,x^4y^6$\\
        $(0,0,0)$&$ x^{10},x^8y^2,x^2y^8,y^{10}$
    \end{tabular}
    \end{center}
    \vspace{0.1in}

    Below is the Hasse diagram of the lattice $C(10,2,2)$:

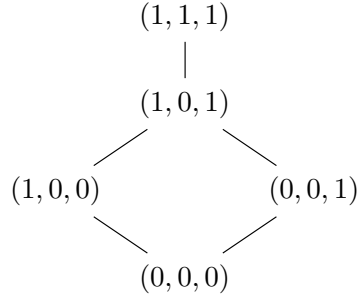
\[\begin{tikzcd}[row sep=3ex, column sep=1ex]
    {} & {(1,1,1)} \arrow[d, dash] & {}\\
    {} & {(1,0,1)}\arrow[dl, dash] \arrow[dr,dash]& {}\\
    {(1,0,0)} \arrow[dr, dash] & {} & {(0,0,1)} \arrow[dl,dash]\\
    {} & {(0,0,0)} & {}
\end{tikzcd}\]

Observe that $C(10,2,2)$ does not include $(1,1,0)$ nor $(0,1,1)$.
On the other hand, $c(x^7y^2) = (0,1,1)\in C(9,2,2)$, exhibiting the dependence of $C(d,n,p)$ on $d$.
\end{ex}

\subsection{Representation theory}\label{sec:rep}
Let $\mathbf{k}$ be an algebraically closed field of characteristic $p\geq 0$, and let $\GL_n(\mathbf{k})$ denote the general linear group of linear automorphisms of $\mathbf{k}^n$. We write $\mathbb{T}_n\subseteq \GL_n(\k)$ for the maximal torus of diagonal matrices. Given a partition $\lambda=(\lambda_1\geq \lambda_2\geq \cdots\geq  \lambda_n\geq 0)\in \mathbb{N}^n$ we write  $L(\lambda)$ for the irreducible representation of $\GL_n(\mathbf{k})$ of highest weight $\lambda$ (see \cite[Section 2.2]{Weyman} where it is written $M_{\lambda}(\k^n)$). The representation $L(\lambda)$ is characterized as the unique irreducible subrepresentation of the indecomposable Schur module $\mathbb{S}_{\lambda}(\mathbf{k}^n)$ \cite[Proposition 2.2.8]{Weyman}. When $\operatorname{char}(\k)=0$ we have $L(\lambda)=\mathbb{S}_{\lambda}(\mathbf{k}^n)$ for all partitions $\lambda$ \cite[Theorem 2.2.10]{Weyman}, but when $\operatorname{char}(\k)=p>0$ we generally have $L(\lambda)\subsetneq \mathbb{S}_{\lambda}(\mathbf{k}^n)$. For instance, when $p=2$, $n=2$, and $\lambda=(2,0)$ we have $\mathbb{S}_{(2)}(\mathbf{k}^2)=S_2=\k\langle x^2, xy, y^2\rangle$ and $L(2,0)=\k \langle x^2, y^2\rangle$. On the other hand, when $\lambda=(1^d,0^{n-d})$ we have $L(1^d)=\mathbb{S}_{(1^d)}(\k^n)=\bigwedge^d(\k^n)$ in every characteristic.

Let $S_d$ be the vector space of homogeneous polynomials of degree $d$ in $S=\mathbf{k}[x_1,\cdots,x_n]$. It is isomorphic as a representation of $\GL_n(\k)$ to the Schur module $\mathbb{S}_{(d)}(\k^n)$, and it is spanned by the monomials of degree $d$:
$$
S_d=\k\langle x_1^{b_1}x_2^{b_2}\cdots x_n^{b_n}\mid b_1+b_2+\cdots +b_n=d\rangle.
$$
The one-dimensional subspace of $S_d$ spanned by the monomial $x^{\underline{b}}$ is the \defi{torus weight space} corresponding to $\underline{b}$. For instance, $x^7y^3z^2$ spans the weight space of $S_{12}$ corresponding to $\underline{b} = (7,3,2)$. We say that a torus weight $\underline{b}$ is a \defi{partition} if $b_1\geq b_2\geq \cdots \geq b_n$. As a representation of $\mathbb{T}_n$, the space $S_d$ decomposes as a direct sum of one-dimensional torus weight spaces. Similarly, every $\GL_n(\k)$-subquotient of $S_d$ decomposes as a direct sum of one-dimensional torus weight spaces \cite[Proposition 2.2.2]{Weyman}. Since $\GL_n(\k)$ contains the permutation matrices $\mathfrak{S}_n$ as a subgroup, the set of torus weights in a subquotient of $S_d$ is closed under permutations. In particular, we can describe a subquotient of $S_d$ by the \textit{partitions} that appear as exponent vectors inside it, which is convenient when $n$ and $d$ are large.  

We recall the $\GL_n(\k)$-submodule lattice of $S_d$ \cite{Doty}. We say a set $B\subseteq C(d,n,p)$ is \defi{order-closed} if for every $c$ in $B$, and every $c'$ in $C(d,n,p)$, if $c'<c$, then $c'$ is in $B$ (an order-closed set is sometimes called an order ideal, but we avoid this terminology in order to avoid confusion with ideals of a ring).  

\begin{ex}
    In Example \ref{degree 10 cp example}, let $B=\{(0,0,0),(1,0,0),(0,0,1)\}$. The set $B$ is an order-closed subset of $C(10,2,2)$, where $(1,0,0)$ and $(0,0,1)$ are both maximal elements.
\end{ex}

Let $B$ be a subset of $C(d,n,p)$.
    Define $T_B$ (or $T_{B,d}$ when we want to emphasize degree) as the $\k$-submodule of $S_d$ generated by the monomials with carry patterns in $B$. 
    
    Given $c \in C(d,n,p)$ we use the shorthand $T_c$ (or $T_{c,d}$) for
    $$
    T_c:=T_{\{c'\in C(d,n,p) \mid c'\leq c\}}.
    $$

We are now ready to state the main theorem of \cite{Doty}, which asserts that for $B\subseteq C(d,n,p)$, the $\k$-submodule $T_B$ of $S_d$ is a $\GL_n(\k)$-submodule of $S_d$ if and only if $B$ is order-closed.

\begin{theorem}[\cite{Doty} Theorem 3]\label{carry-submodule correspondence}
    The correspondence $B\to T_B$ defines a lattice isomorphism between the lattice of order-closed subsets of $C(d,n,p)$ and the lattice of $\GL_n(\k)$-submodules of $S_d$.
\end{theorem}

\begin{ex}

    For $n=2$ and $p=2$ we consider $\GL_2(\k)$-submodules of $S_{10}$ corresponding to order-closed subsets of $C(10,2,2)$. We index each order-closed subset of $C(10,2,2)$ by its maximal carry patterns:
    \vspace{0.1in}
    \begin{center}

    \begin{tabular}{c|c}
            maximal carry patterns &$\GL_n(\k)-$submodule of $S_{10}$\\
        \hline
        (1,1,1)& $S_{10}$\\
        (1,0,1)& $\langle x^{10},x^9y,x^8y^2,x^6y^4,x^5y^5,x^4y^6,x^2y^8,xy^9,y^{10}\rangle$\\
        (1,0,0),\;(0,0,1)&$\langle y^{10}, x^9y,x^8y^2,x^6y^4,x^4y^6,x^2y^8,xy^9,y^{10}\rangle$\\
        (1,0,0)&$\langle x^{10}, x^9y,x^8y^2,x^2y^8,xy^9,y^{10}\rangle$\\
        (0,0,1)&$\langle x^{10},x^8y^2,x^6y^4,x^4y^6,x^2y^8,y^{10}\rangle$\\
        (0,0,0)&$\langle x^{10}, x^8y^2,x^2y^8,y^{10}\rangle$
    \end{tabular}
    \end{center}

    \smallskip
    
    The submodule in the third row is $T_{\{(0,0,0),(1,0,0),(0,0,1)\}}=T_{(1,0,0)}+T_{(0,0,1)}$.

    Examining torus weights (by considering degree sequences of monomials) in each row of the table above, we see that $S_{10}$ has composition factors $L(10,0)$, $L(6,4)$, $L(9,1)$, $L(5,5)$, $L(7,3)$, each with multiplicity one. More precisely, $S_{10}$ is characterized by the following non-split short exact sequences:
    $$
    0\longrightarrow T_{(1,0,1)}\longrightarrow S_{10}\longrightarrow L(7,3)\longrightarrow 0,\quad\quad  0\longrightarrow T_{(1,0,0)}+T_{(0,0,1)}\longrightarrow  T_{(1,0,1)}\longrightarrow L(5,5)\longrightarrow 0
    $$
    $$
    0\longrightarrow L(10,0)=T_{(0,0,0)}\longrightarrow T_{(1,0,0)}+T_{(0,0,1)}\longrightarrow L(6,4)\oplus L(9,1)\longrightarrow 0.
    $$
    \end{ex}
In general, if $B'\subseteq B$ with $|B|=|B'|+1$ then $T_{B}/T_{B'}\cong L(\lambda)$, where $x^{\lambda}$ is the lexicographically maximal monomial appearing in $T_{B}/T_{B'}$ (see \cite[Proposition 2.2.3, Theorem 2.2.9]{Weyman}). 

As a consequence of Corollary \ref{mincarry} and Proposition \ref{prop:lcm is join} we recover the following well-known result.

\begin{cor}\label{cor:simplesub}
The Schur module $S_{d}$ has  unique simple submodule $L(d)$ (corresponding to carry pattern $(0,\cdots,0)$), and unique simple quotient corresponding to the unique maximal carry pattern.
\end{cor}

\begin{ex}\label{exsmalld}
(1) If $d\leq p-1$ then $C(d,n,p)=\{(0,\cdots,0)\}$ and $S_d=L(d)$. 

(2) If $d=p$ then $L(p)=\k\langle x_1^p,x_2^p,\cdots,x_n^p\rangle \subsetneq S_p$.
\end{ex}

See Corollary \ref{cor:simple2} for a characterization of irreducibility of $S_d$ when the number of variables is $n=2$. 

\section{Invariant ideals and their properties}\label{section: invariant ideals}

Let $\k$ be an algebraically closed field of characteristic $p>0$ and let $S = \k[x_1,\dots,x_n]$. Given a subgroup $G\subseteq \GL_n(\k)$, we say that an ideal $I$ in $S$ is $G$-invariant if $g\cdot f\in I$ for all $g\in G$ and all $f\in I$. In this section, we classify $\GL_n(\k)$-invariant ideals and study their generators. In what follows, we often refer to  $\GL_n(\k)$-invariant ideals  simply as ``invariant ideals''.

\subsection{Classification of invariant ideals}\label{sec:class} In this subsection we use the representation theory developed in Section \ref{section: Carry Patterns} to classify invariant ideals in $S$. We freely use notation and terminology surrounding carry patterns and the submodule lattice of a graded component $S_d\subseteq S$.

\begin{defn}
Given $d\geq 0$ and $c\in C(d,n,p)$ we define the \defi{carry ideal} $I_{c,d}\subseteq S$ associated to $c$ via
$$
I_{c,d}=S\langle T_{c,d} \rangle.
$$

In other words, it is the ideal generated by monomials $x^{\underline{b}}$ of degree $d$ with $c(\underline{b})\leq c$.
\end{defn}

As an ideal $I=\langle m_1,\cdots, m_s\rangle\subseteq S$ is invariant if and only if $g\cdot m_i\in I$ for all $g\in \GL_n(\k)$ and all $1\leq i\leq s$, it follows that carry ideals are invariant.  Furthermore, it follows from Corollary \ref{cor:simplesub} that $I_{c,d}$ contains $T_{(0,\ldots,0),d}$ as a subrepresentation. Hence, every carry ideal generated in degree $d$ contains the powers $x_1^d,\cdots, x_n^d$ (and every monomial in $L(d)\subseteq S_d$). 

\begin{ex}\label{ex:somecarryideals}
Below are some examples of carry ideals.
\begin{enumerate}
    \item If $n=2$, $p=2$, $d=2$, and $c=(0)$, then $I_{c,d}=\langle x^2,y^2\rangle$. 
    \item If $n=2$, $p=5$, $d=25$, and $c=(0,1)$, then
    $I_{c,d} = \langle x^{25},\:x^{20}y^{5},\:x^{15}y^{10},\:x^{10}y^{15},\:x^{5}y^{20},\:y^{25}\rangle$.
    \item If $n=3$, $p=5$, $d=35$, and $c = (2,0)$, then minimal generators of $I_{c,d}$ are the monomials with exponent vectors given by a permutation of the following partitions:
    \begin{center}
        \begin{tabular}{cccccccccc}
$(35, 0, 0)$,& $(34, 1, 0)$,& $(33, 2, 0)$,& $(33, 1, 1)$,& $(32, 3, 0)$,& $(32, 2, 1)$,&
$(31, 4, 0)$,& $(31, 3, 1)$,& \\
$(31, 2, 2)$,& $(30, 5, 0)$,& $(30, 4, 1)$,& $(30, 3, 2)$,&
$(29, 6, 0)$,& $(29, 5, 1)$,& $(29, 4, 2)$,& $(29, 3, 3)$,& \\
$(28, 7, 0)$,& $(28, 6, 1)$,& $(28, 5, 2)$,& $(28, 4, 3)$,& $(27, 8, 0)$,& $(27, 7, 1)$,& $(27, 6, 2)$,& $(27, 5, 3)$,&\\
$(27, 4, 4)$,& $(26, 9, 0)$,& $(26, 8, 1)$,& $(26, 7, 2)$,& $(26, 6, 3)$,& $(26, 5, 4)$,&
$(25, 10, 0)$,& $(25, 9, 1)$,&\\
$(25, 8, 2)$,& $(25, 7, 3)$,& $(25, 6, 4)$& $(25, 5, 5)$.
        \end{tabular}
    \end{center}
    so that, for example, the monomials $x^{35}$, $xy^{31}z^3$, $x^{4}y^{5}z^{26}$ are generators of $I_{c,d}$. We remark that
    $d=35$ is minimal such that there exists $c\in C(d,3,5)$ satisfying $c_i>1$ for some $i$ and $T_{c,d}\neq S_d$.
    For example, the monomials $x^{24}y^{11}$ and $x^5y^8z^{22}$ are degree $35$ monomials which are not in $I_{c,d}$.
    
    \item Let $c$ be the maximal carry pattern in $C(d,n,p)$ (see Corollary \ref{cor:simplesub}). Then $I_{c,d}=\langle x_1,\cdots,x_n\rangle^d$.
\end{enumerate}
\end{ex}

It follows from Theorem \ref{carry-submodule correspondence} that
\begin{equation}\label{eq:containsamedegree}
I_{c',d}\subseteq I_{c,d} \quad \iff \quad c'\leq c.    
\end{equation}

Next we define the more general notion of an ideal corresponding to a collection of carry patterns. Given a fixed $n\geq 1$ and $p>0$ we define
$$
C(n,p)=\bigsqcup_{d\geq 0}C(d,n,p),
$$
and for a finite subset $B$ of $C(n,p)$, we write $B_d=B\cap C(d,n,p)$, so that
$$
B=\bigsqcup_{d\geq 0} B_d.
$$
With this notation, we make the following definition.
\begin{defn}\label{def:invarianti}
Given a finite subset $B$ of $C(n,p)$, we define the $S$-ideal
$$
I_B=\sum_{d\geq 0}\sum_{c\in B_d}I_{c,d}.
$$
\end{defn}
Since $I_B$ is a sum of carry ideals, it is a $\GL_n(\k)$-invariant ideal. It follows from Theorem \ref{carry-submodule correspondence} that, if we write $B_d^{\max}$ for the set of maximal elements in $B_d$, then we have an equality
\begin{equation}\label{def:invideal}
I_B=\sum_{d\geq 0}\sum_{c\in B_d^{\max}}I_{c,d}.
\end{equation}
In particular, if $B,B'\subseteq C(n,p)$ and $B_d^{\max}=(B'_d)^{\max}$ for all $d\geq 0$, then $I_B=I_{B'}$. We observe that the sum (\ref{def:invideal}) may have redundancies if there are $d<d'$, $c\in B_d^{\max}$, and $c'\in B_{d'}^{\max}$ with $I_{c',d'}\subseteq I_{c,d}$. In Section \ref{sec:contain} we explore the conditions on $c,c'$ that guarantee $I_{c',d'}\subseteq I_{c,d}$.

It follows immediately from the definition that if $B,B'\subseteq C(n,p)$ then
\begin{equation}
    I_{B\cup B'}=I_B+I_{B'}.
\end{equation}

Next, we show that $S/I$ has finite length when $I$ is a nonzero invariant ideal.

\begin{lemma}\label{lemma:origin}
If $I$ is a nonzero invariant ideal, then $I_d=S_d$ for some $d\gg 0$. In particular, $\sqrt{I}=\langle x_1,\cdots,x_n\rangle$.   
\end{lemma}

\begin{proof}
Let $I_{d_1}$ denote the lowest-degree non-zero homogeneous component of $I$. 
    Observe that every monomial of degree $nd_1$ is divisible by $x_i^{d_1}$ for some index $i$. Thus, $I_{n d_1}$ contains all monomials of degree $n d_1$ since $I_{d_1}$ contains $x_i^{d_1}$ for each index $i$. In other words, for any choice of invariant ideal $I$, there is always an index, namely $nd_1$ such that $I_{nd_1} = S_{nd_1}$.

    Since $x_i^{d_1}\in I$ for all $1\leq i\leq n$, it follows that $x_1,\cdots, x_n\in \sqrt{I}$. Since $I$ is homogeneous it follows that $\sqrt{I}=\langle x_1,\cdots,x_n\rangle$.
\end{proof}

A more geometric argument is that $S/I$ is necessarily supported on the closure of a $\GL_n(\k)$-orbit of $\mathbb{A}^n_{\k}$. Thus, if $I\neq 0$ then $S/I$ is supported on the origin.

We are now ready to state the main result of the section.

\begin{theorem}[Classification of Invariant Ideals]
    \label{structure theorem}
    Every $\GL_n(\k)$-invariant ideal $I$ in $S=\k[x_1,\cdots, x_n]$ is equal to $I_B$ for some finite $B\subseteq C(n,p)$.
\end{theorem}

The following proof gives a procedure to find $B$ with no redundancies. 

\begin{proof}

    For each $d\geq 0$ the graded piece $I_d$ of $I$ is a subrepresentation of $S_d$. We define $X_d\subseteq C(d,n,p)$ to be the set of carry patterns of monomials in $I_d$, and define $Y_d\subseteq X_d$ to be the set of carry patterns of monomials in $S_1\cdot I_{d-1}$. We set
    $$
    B=\bigsqcup_{d\geq 0}B_d,\quad\textnormal{where}\quad B_d=(X_d\setminus Y_d)^{\max}.
    $$
    Then $I_d=T_{X_d}$ and we have
    $$
    I=\sum_{d\geq 0}\sum_{c\in X_d}I_{c,d}=\sum_{d\geq 0}\sum_{c\in X_d\setminus Y_d}I_{c,d}=\sum_{d\geq 0}\sum_{c\in B_d}I_{c,d},
    $$
 where the second equality follows from Theorem \ref{carry-submodule correspondence} and the definition of $Y_d$, and the third equality follows from $(\ref{def:invideal})$. The latter sum is necessarily finite by Lemma \ref{lemma:origin}. 
    
To see that $B$ has no redundancies, suppose that $c\in B_d$ and that $I_{c',d'}\subseteq I_{c,d}$ for some $d'>d$. In this case, the set of carry patterns of generators of $I_{c',d'}$ belongs to $Y_{d'}$, so $c'$ does not belong to $B_{d'}$.   
\end{proof}

It follows from Corollary \ref{cor:simplesub}  that if $Y_d\neq \emptyset$ then $(0,\cdots,0)\in Y_d$.

\begin{ex}
    Let $S=\k[x,y]$ where $\Char(\k)=2$. Consider the invariant ideal
    $$
    I=\langle x^8, x^7y^3, x^5y^4, x^4y^5, x^3y^7, y^8\rangle,
    $$
    which has generators of degrees $8$, $9$, and $10$. The graded pieces of $I$ in those degrees are
    $$
    I_8=\k\langle x^8,y^8\rangle,\quad I_9=\k\langle x^9,x^8y,x^5y^4,x^4y^5,xy^8,y^9\rangle,\quad I_{10}=S_{10}.
    $$
    Using notation from the proof of Theorem \ref{structure theorem} we have that
    $$
    X_8=\{(0,0,0)\},\quad X_9=\{(0,0,0),(0,0,1)\},\quad Y_9=\{(0,0,0)\},
    $$
    $$
    X_{10}=\{(0, 0, 0), (0, 0, 1), (1, 0, 0), (1, 0, 1), (1, 1, 1)\}, \quad Y_{10}=X_{10}\setminus \{(1,1,1)\}. 
    $$
    We conclude that $I=I_{(0,0,0),8}+I_{(0,0,1),9}+I_{(1,1,1),10}$.
\end{ex}

\begin{remark}
Since $\GL_n(\k)$-invariant ideals are examples of  $\mathfrak{S}_n$-invariant monomial ideals, one could also index them by partitions parametrizing the $\mathfrak{S}_n$-orbits of the generators. For example, $I=I_{\{(8,0),(5,4)\}}=\langle x^8, x^5y^4, x^4y^5, y^8\rangle$. Not all $\mathfrak{S}_n$-invariant monomial ideals are $\GL_n(\k)$-invariant (e.g. $I_{\{(1,1)\}}=\langle xy\rangle\subseteq \k[x,y]$ is not $\GL_2(\k)$-invariant in any characteristic). The partitions appearing encode algebraic properties of the ideal. For instance, if $\lambda$ and $\mu$ are partitions with at most $n$ parts, then $I_{\lambda}\subseteq I_{\mu}$ in $\k[x_1,\cdots,x_n]$ if and only if $\lambda_i\leq \mu_i$ for $i=1,\cdots, n$. In \cite{raicu}, Castelnuovo--Mumford regularity of $\mathfrak{S}_n$-invariant monomial ideals is calculated using combinatorics of the partitions appearing as exponents.
\end{remark}

\subsection{Products, powers, and Frobenius powers of invariant ideals}\label{sec:powers}

In this section we investigate various operations on invariant ideals. We begin by considering products and powers.

\begin{prop}
Given two invariant ideals $I,J\subseteq S$, their product $IJ$ is also an invariant ideal. In particular, powers of invariant ideals are invariant.
\end{prop}

\begin{proof}
Write $I=\langle m_1,\cdots,m_s\rangle$ and $J=\langle m_1',\cdots, m_t'\rangle$ where the $m_i$'s and $m_j'$'s are monomials. Then
$$
IJ=\langle m_im_j'\mid 1\leq i\leq s, \; 1\leq j\leq t\rangle.
$$
We want to show that $g\cdot m_im_j'\in I$ for all $g\in \GL_n(\k)$ and all $i,j$. Since $I$ and $J$ are invariant, we have that
$$
g\cdot m_i=a_{i,1}m_1+\cdots+a_{i,s}m_s,\quad g\cdot m_j'=a_{j,1}'m_1'+\cdots+a_{j,t}'m_t',
$$
for some $a_{i,k},a_{j,l}'\in \k$. So 
$$
g\cdot m_im_j'=(g\cdot m_i)(g\cdot m_j')=\sum_{k,l}a_{i,k}a_{j,l}'m_km_l'.
$$
In particular, $g\cdot m_im_j'\in IJ$, as required to complete the proof.
\end{proof}

Let $x^{\underline{b}}$ and $x^{\underline{b'}}$ be two monomials. Since it is often the case that $c(\underline{b+b'})\neq c(\underline{b})+c(\underline{b'})$, there is no obvious formula for the carry patterns that appear in a product of invariant ideals in terms of the carry patterns of the constituents. 

\begin{ex}
Below are some examples of products and powers of invariant ideals.
\begin{enumerate}
    \item If $n=3$, $p=3$, and $I=I_{(0),1}=\langle x,y,z\rangle$, then $I^3=\langle x^3,x^2y,x^2z,xy^2,xyz,xz^2,y^2z,yz^2, z^3\rangle=I_{(1,0),3}$.

    \item Let $n=2$, $p=7$, $I=I_{(0),7}=\langle x^7,y^7\rangle$, $J=I_{(0),3}=\langle x^3, x^2y, xy^2, y^3\rangle$. Then
    $$
IJ=\langle x^{10},x^9y, x^8y^2, x^7y^3, x^3y^7,x^2y^8,xy^9,y^{10}\rangle=I_{(0),10}.
    $$
\end{enumerate}

\end{ex}

Next, we investigate the implications of the Frobenius endomorphism on the invariant ideals. Let $I = \langle f_1,\dots,f_r\rangle$ be an ideal of $S$ and let $e\geq 1$ be an integer. The $p^e$-th \defi{Frobenius power} of $I$ is the ideal $I^{[p^e]} =  \langle f_1^{p^e},\dots,f_r^{p^e} \rangle$. The next result asserts that, unlike products and powers, Frobenius powers play nicely with carry patterns. Given $c\in C(d,n,p)$ and $e\geq 1$ we define
\begin{equation}
c^{p^e}=(\underbrace{0,\cdots, 0}_{e},c_1, \cdots,c_M)\in C(dp^e,n,p).    
\end{equation}
Observe that if $x^{\underline{b}}=x_1^{b_1}\cdots x_n^{b_n}$ has carry pattern $c=c(\underline{b})$, then $x^{\underline{b}p^e}=x_1^{b_1p^e}\cdots x_n^{b_np^e}$ has carry pattern $c(\underline{b}p^e)=c^{p^e}$. As a consequence, we obtain the following result.

\begin{prop}
For $B\subseteq C(n,p)$ and $e\geq 1$ we have $(I_B)^{[p^e]}=I_{B^{[p^e]}}$ where
$$
B^{[p^e]}=\bigsqcup_{d\geq 0}\{c^{p^e} \mid c\in B_d\}.
$$
In particular, $(I_{c,d})^{[p^e]}=I_{c^{p^e},dp^e}$.
\end{prop}

\begin{ex}
Below are a couple examples involving Frobenius powers.
\begin{enumerate}
    \item In Example \ref{ex:somecarryideals}(2), we had $p=5$ and $I_{c,d}=I_{(1),5}^{[5]}$.

    \item Let $a_1,\cdots,a_s, e_1,\cdots, e_s\geq 1$. Since $\m=\langle x_1,\cdots, x_n\rangle$ is an invariant ideal, so is
    $$
I=(\m^{a_1})^{[p^{e_1}]}(\m^{a_2})^{[p^{e_2}]}\cdots (\m^{a_s})^{[p^{e_s}]}.
    $$
    In Section \ref{gens2} we will see that all carry ideals in two variables are of this form.
\end{enumerate}

\end{ex}

\subsection{Containment of carry ideals}\label{sec:contain} In this subsection, we investigate the following problem.

\begin{prob}\label{containprob}
Let $c\in C(d,n,p)$ and let $d'\geq d$. Characterize $c'\in C(d',n,p)$ for which $I_{c,d}\supseteq I_{c',d'}$.   
\end{prob}

A solution to this problem would allow us to remove all redundancy in the definition of the ideals $I_B$ for $B\subseteq C(n,p)$ (see Definition \ref{def:invarianti}), and obtain a one-to-one correspondence between $\GL_n(\k)$-invariant ideals in $\k[x_1,\cdots,x_n]$ and certain subsets of $C(n,p)$.

If $d'=d$, then by (\ref{eq:containsamedegree}) we have $I_{c,d}\supseteq I_{c',d'}$ if and only if $c\geq c'$. On the other hand, for $d'>d$, this is no longer the case: we may have $c'\nleq c$ but $I_{c',d'}\subseteq I_{c,d}$, as demonstrated by the following example.

\begin{ex}
If $d=1$ and $c=(0,\cdots,0)$, then $I_{c,d}=\langle x_1,\cdots, x_n\rangle$, so $I_{c,d}\supseteq I_{c',d'}$ for all pairs $(c',d')$, but there are $c'>(0,\cdots,0)$. For instance, $(1)\in C(2,2,2)$.

\end{ex}

Using notation from Section \ref{sec:rep}, we are interested in the image of the multiplication map
\begin{equation}\label{eq:multmap}
S_{d'-d}\otimes T_{c,d}\longrightarrow S_{d'}.
\end{equation}
Indeed, $I_{c,d}\supseteq I_{c',d'}$ if and only if $T_{c',d'}$ is a subrepresentation of the image of (\ref{eq:multmap}). An important step in determining the image of this multiplication map is to understand the case $d'-d=1$, as one could then use induction to address the general case. Thus, we are interested in the carry patterns of $x_ix^{\underline{b}}$ for $x^{\underline{b}}\in T_{c,d}$.

To state our main theorem in this direction, we introduce some notation. We recall from (\ref{eq:floor}) that for an integer $a\geq 0$ with base-$p$ expansion $a=(a_0,\cdots, a_N)$, we write $\fl{a}=\min\{j\mid a_j\neq p-1\}$. We continue to write $d=(d_0,\cdots, d_M)$ for the base-$p$ expansion of $d$, with $d_M\neq 0$.  Given $c\in C(d,n,p)$, consider any monomial \(x^{\underline b} = x_1^{b_1}\cdots x_n^{b_n}\) of degree \(d\) with carry pattern \(c\), and write $b_i=(b_{i0},\cdots ,b_{iM})$ for the base-$p$ expansion of $b_i$ for $1\leq i \leq n$. By the definition of carry pattern (\ref{carry pattern}) we have for each \(0\leq j\leq M\): 
\begin{equation}
    d_j + pc_{j+1} - c_j = b_{1j}+\dots + b_{nj}\leq n(p-1). 
    \end{equation}
As such, we define $c^{\sharp}$ as follows
\begin{equation}
c^\sharp := \min\{0\leq j\leq M|d_j + pc_{j+1} - c_j < n(p-1)\}.    
\end{equation}

\begin{theorem}\label{thm:multmap}
Let $c\in C(d,n,p)$. The image of the multiplication map \(S_1\otimes T_{c,d}\to S_{d+1}\) is \(T_{c^1,d+1}\), where \(c^1\in C(d+1,n,p)\) is defined by 
$$
c^1=(c_1,\cdots,c_M,0)+(\underbrace{1,\cdots, 1}_{\fl{d}},\underbrace{0,\cdots,0}_{M+1-\fl{d}})-(\underbrace{1,\cdots, 1}_{c^\sharp}, \underbrace{0\cdots, 0}_{M+1-c^\sharp}).
$$
\end{theorem}

Theorem \ref{thm:multmap} is a consequence of the following (stronger) technical result.

\begin{theorem}\label{carry multiply x_i induction} Let $c=(c_1,\cdots,c_M)$ be an element of $C(d,n,p)$.
\begin{enumerate}
    \item Let $x^{\underline{b}}=x_1^{b_1}\cdots x_n^{b_n}$ be a monomial of degree $d$ with carry pattern $c$. The carry pattern of $x_ix^{\underline{b}}$ is
$$
c(x_ix^{\underline{b}})=(c_1,\cdots,c_M,0)+(\underbrace{1,\cdots, 1}_{\fl{d}},\underbrace{0,\cdots,0}_{M+1-\fl{d}})-(\underbrace{1,\cdots, 1}_{\fl{b_i}}, \underbrace{0\cdots, 0}_{M+1-\fl{b_i}}).
$$

\item Using notation as in the statement of Theorem \ref{thm:multmap}, $c^1$ is the maximal element in $C(d+1,n,p)$ obtained by multiplying a monomial of degree $d$ and carry pattern $c$ by a variable.

\item If \(c'\in C(d,n,p)\) with \(c'< c\), then \(c'^1\leq c^1\) in \(C(d+1,n,p)\).
\end{enumerate}
\end{theorem}

\begin{proof} (1) Let $c'=c(x_ix^{\underline{b}})$, and for each $1\leq j\leq n$ let $b_j=(b_{j,0},\cdots,b_{j,M})$ denote the base-$p$ expansion of $b_j$. The defining equations (\ref{carry pattern}) of carry patterns give the following relations for $1\leq k \leq M+1$:
$$
        \sum_{j=1}^n b_{j,k-1} = d_{k-1} + pc_{k} - c_{k-1},\quad \sum_{j=1}^n b_{j,k-1}-b_{i,k-1}+(b_i+1)_{k-1}=(d+1)_{k-1}+pc_k'-c'_{k-1}.
        $$
Taking the difference of these two expressions, and isolating $p(c_k'-c_k)$, we obtain the following:
\begin{equation}\label{carrydifference}
p(c_k'-c_k)=((b_i+1)_{k-1}-b_{i,k-1})+(d_{k-1}-(d+1)_{k-1})+(c'_{k-1}-c_{k-1}).   
\end{equation}
For ease of notation, we set $l=\fl{b}$ and $m=\fl{d}$, so that $b_i$ and $b_i+1$ have base-$p$ expansion:
$$
b_i=(\underbrace{p-1,\cdots,p-1}_{l-1},b_{i,l},\cdots,b_{i,M}),\quad b_i+1=(\underbrace{0,\cdots,0}_{l-1},b_{i,l}+1,b_{i,l+1},\cdots, b_{i,M}),
$$
and similarly, $d$ and $d+1$ have base-$p$ expansion
$$
d=(\underbrace{p-1,\cdots,p-1}_{m-1},d_{m},\cdots,d_{M}),\quad d+1=(\underbrace{0,\cdots,0}_{m-1},d_{m}+1,d_{m+1},\cdots, d_{M}).
$$
We break the proof into three cases, depending on whether $m >l$, $m=l$, or $m<l$. 

\textit{Case 1.} Assume $m> l$, so we wish to show that
$$
c'=(c_1,\cdots,c_M,0)+(\underbrace{0,\cdots,0}_{l},\underbrace{1,\cdots, 1}_{m-l},\underbrace{0,\cdots,0}_{M+1-m}).
$$
We induct on $k\geq 0$ to show that $c'_k$ is as claimed. The base case $c_0'=0$ is vacuous. If $1\leq k\leq l$, then by (\ref{carrydifference}) we have $p(c_k'-c_k)=(1-p)+(p-1)+(c_{k-1}'-c_{k-1})$, so by the inductive hypothesis ($c_{k-1}'-c_{k-1}=0$) it follows that $c_k'=c_k$. Next, if $k=l+1$, then since $m>l$ we have that  (\ref{carrydifference}) gives $p(c_k'-c_k)=1+(p-1)+(c_{k-1}'-c_{k-1})$, so by the inductive hypothesis ($c_{k-1}'-c_{k-1}=0$) we conclude that $c_k'-c_k=1$. If $l+2\leq k\leq m$ then by (\ref{carrydifference}) we have $p(c_k'-c_k)=0+(p-1)+(c_{k-1}'-c_{k-1})$, so by the inductive hypothesis ($c_{k-1}'-c_{k-1}=1$) we obtain $c_k'-c_k=1$. Next, if $k=m+1$ then we have that $p(c_k'-c_k)=0+(-1)+(c_{k-1}'-c_{k-1})$, so by the inductive hypothesis ($c_{k-1}'-c_{k-1}=1$) we conclude that $c_k'=c_k$. Finally, if $k\geq m+2$ then by (\ref{carrydifference}) we have $p(c_k'-c_k)=0+0+(c_{k-1}'-c_{k-1})$, so by the inductive hypothesis ($c_{k-1}'-c_{k-1}=0$) we have that $c'_k=c_k$.

\textit{Case 2.} Assume $m=l$, so we wish to show that $c=c'$. Again, we proceed by induction on $k\geq 0$, the base case $k=0$ being vacuous. If $1\leq k\leq m$, then by (\ref{carrydifference}) we have $p(c_k'-c_k)=(1-p)+(p-1)+(c_{k-1}'-c_{k-1})$, so by the inductive hypothesis ($c_{k-1}'-c_{k-1}=0$) it follows that $c_k'=c_k$. Next, if $k=m+1$, then since $m=l$ we have that  (\ref{carrydifference}) gives $p(c_k'-c_k)=1+(-1)+(c_{k-1}'-c_{k-1})$, so by the inductive hypothesis ($c_{k-1}'-c_{k-1}=0$) we conclude that $c_k'=c_k$. Finally, if $k>m$ then by (\ref{carrydifference}) we obtain that $p(c_k'-c_k)=0+0+(c_{k-1}'-c_{k-1})$, so by the inductive hypothesis ($c_{k-1}'-c_{k-1}=0$) it follows that $c_k'=c_k$. 

\textit{Case 3.} Assume $m<l$, so we wish to show that
$$
c'=(c_1,\cdots,c_M,0)+(\underbrace{0,\cdots,0}_{m},\underbrace{-1,\cdots, -1}_{l-m},\underbrace{0,\cdots,0}_{M+1-l}).
$$
We proceed by induction on $k\geq 0$ as in the previous cases. If $1\leq k\leq m$, then by (\ref{carrydifference}) we have $p(c_k'-c_k)=(1-p)+(p-1)+(c_{k-1}'-c_{k-1})$, so by the inductive hypothesis ($c_{k-1}'-c_{k-1}=0$) it follows that $c_k'=c_k$. Next, if $k=m+1$, then since $m<l$ we have that  (\ref{carrydifference}) gives $p(c_k'-c_k)=(1-p)+(-1)+(c_{k-1}'-c_{k-1})$, so by the inductive hypothesis ($c_{k-1}'-c_{k-1}=0$) we conclude that $c_k'-c_k=-1$. If $m+2\leq k\leq l$ then by (\ref{carrydifference}) we have $p(c_k'-c_k)=(1-p)+0+(c_{k-1}'-c_{k-1})$, so by the inductive hypothesis ($c_{k-1}'-c_{k-1}=-1$) we obtain $c_k'-c_k=-1$. Next, if $k=l+1$ then we have that $p(c_k'-c_k)=1+0+(c_{k-1}'-c_{k-1})$, so by the inductive hypothesis ($c_{k-1}'-c_{k-1}=-1$) we conclude that $c_k'=c_k$. Finally, if $k\geq l+2$ then by (\ref{carrydifference}) we have $p(c_k'-c_k)=0+0+(c_{k-1}'-c_{k-1})$, so by the inductive hypothesis ($c_{k-1}'-c_{k-1}=0$) we have that $c'_k=c_k$.

   (2) By part (1), it suffices to show that \(c^\sharp\) equals the minimal \(b_i^\circ\) we can get from some monomial \(x^{\underline b} = x_1^{b_1}\cdots x_n^{b_n}\) of degree \(d\) with carry pattern \(c\). Since carry patterns are invariant under permutations of the variables, it suffices to minimize \(b_1^\circ\), so we want to show that \(c^\sharp\) is equal to the number
    \[
    \min\{0\leq j\leq M|\;\underline{b}\in \mathbb{Z}^n_{\geq 0},\;\deg(x^{\underline b}) = d,\; c(x^{\underline{b}}) = c,\; b_{1j}\neq p-1\}.
    \]
Let $\underline{b}\in \mathbb{Z}^n_{\geq 0}$, and write $b_i=(b_{i0},\cdots ,b_{iM})$ for the base-$p$ expansion of $b_i$ for $1\leq i \leq n$. Then $x^{\underline{b}}$ has degree $d$ and carry pattern $c$ if and only if for each \(0\leq j\leq M\), we have
\begin{equation}\label{keyeq}
    d_j + pc_{j+1} - c_j = b_{1j}+\dots + b_{nj}. 
    \end{equation}
By (\ref{keyeq}) we can choose \(b_{1j}\) to be less than \(p-1\) if and only if \(d_j + pc_{j+1}-c_j\neq n(p-1)\). Therefore, the minimal \(b_1^\circ\) is equal to \(c^\sharp\).

    (3)  Let $c'\in C(d,n,p)$ with $c'<c$ and let \(i\) be the first index where \(c'_i<c_i\). Since
    \begin{equation}\label{keyeq2}
    d_{i-1} + pc'_i - c'_{i-1}< d_{i-1} + pc_i - c_{i-1}\leq n(p-1), 
    \end{equation}
    we know \(0\leq (c')^\sharp \leq i-1\). 
    
    Suppose \((c')^\sharp<i-1\). For \(0\leq j\leq i-2\), we have
    \[
    d_j + pc'_{j+1} - c'_j = d_j + pc_{j+1} - c_{j}.
    \]
    So in this case, \((c')^\sharp = c^\sharp\), and by part (2) we are done.

    Suppose now that \((c')^\sharp = i-1\). By (\ref{keyeq2}) we have \(c^\sharp \geq (c')^\sharp = i-1\), so by part (2) it suffices to show that
    \[
    (c'_1,\cdots,c'_M,0)-(\underbrace{1,\cdots, 1}_{(c')^\sharp}, \underbrace{0,\cdots, 0}_{M+1-(c')^\sharp}) \leq (c_1,\cdots,c_M,0)-(\underbrace{1,\cdots, 1}_{c^\sharp}, \underbrace{0,\cdots, 0}_{M+1-c^\sharp}).
    \]
    Rearranging terms, this is equivalent to
    \[
    (c_i-c'_i,\dots,c_M - c'_M,0)\geq (\underbrace{1,\cdots, 1}_{c^\sharp-(i-1)}, \underbrace{0\cdots, 0}_{M+1-c^\sharp}).
    \]
    Let $k$ be the first index with \(i+1\leq k\leq M+1\) and \(c'_k = c_k\). We need to show that $k-i\geq c^\sharp - (i-1)$, or equivalently, $k - 1 \geq c'^\sharp$. Since $c_k'=c_k$, by minimality of $k$ we have
    $$
    d_{k-1} + pc_k - c_{k-1} = d_{k-1} + pc'_k - c_{k-1} < d_{k-1} + pc'_k - c'_{k-1}\leq n(p-1).
    $$
    Thus, \(c^\sharp  \leq k-1\), as claimed.
\end{proof}

\begin{ex}
Let $p=7$ and $d=440=(6,6,1,1)$, so that $\fl{d}=2$. We consider $x^{342}y^{48}z^{50}\in \k[x,y,z]_d$, which has carry pattern $c=(1,1,1)$, as $342=(6,6,6,0)$, $48=(6,6,0,0)$, and $50=(1,0,1,0)$. Since $\fl{342}=3$, $\fl{48}=2$, and $\fl{50}=0$, Theorem \ref{carry multiply x_i induction}(1) asserts that
$$
c(x^{343}y^{48}z^{50})=(1,1,1)+(1,1,0)-(1,1,1)=(1,1,0),\quad  c(x^{342}y^{49}z^{50})=(1,1,1)+(1,1,0)-(1,1,0)=(1,1,1),
$$
$$
c(x^{342}y^{48}z^{51})=(1,1,1)+(1,1,0)-(0,0,0)=(2,2,1),
$$
which can be verified directly using (\ref{carry pattern}) and the exponents involved.
\end{ex}

\subsection{Carry ideals in two variables}\label{gens2} In this section we set $n=2$, so that $S=\k[x,y]$. In this case, the degree $d$ of a monomial and the power on $x$ determine the monomial uniquely, as any degree $d$ monomial is of the form $x^ay^{d-a}$. As a consequence, we have more control on the possible carry patterns in $C(d,2,p)$. 

For the following statement, we recall  $\fl{d}=\min\{j\mid d_j\neq p-1\}$ and $M=\max\{j\mid d_j\neq 0\}$, where $d=(d_0,\cdots,d_M)$ is the base-$p$ expansion of an integer $d\geq 0$.

\begin{prop}\label{prop: carry only 0 or 1}
The following is true about elements of $C(d,2,p)$:
\begin{enumerate}
    \item A sequence $c=(c_1,\cdots, c_M)$ belongs to $C(d,2,p)$ if and only if the following conditions hold:
    \begin{enumerate}[(i)]
        \item $c_i\in \{0,1\}$ for all $i=1,\cdots, M$,
        \item if $d_i=0$ for some $i$, then $c_{i+1}\geq c_i$,
        \item if $d_i=p-1$ for some $i$, then $c_i\geq c_{i+1}$.
    \end{enumerate}
    \item If $c\in C(d,2,p)$ then $c_1=\cdots=c_{\fl{d}}=0$.
    \item If $M\neq 0$ the maximal element in $C(d,2,p)$ is
    $$
    c=(c_1,\cdots, c_M)=(\underbrace{0,\cdots,0}_{\fl{d}},\underbrace{1,\cdots,1}_{M-\fl{d}}).
    $$
    \item If $M=0$ then $C(d,2,p)=\{(0)\}$.
\end{enumerate}    
\end{prop}
\begin{proof}
First assume that $c\in C(d,2,p)$. We will show that $c_i\in\{0,1\}$ for all $i$, and verify assertion (2) that $c_1=\cdots=c_{\fl{d}}=0$. We prove by induction on $i\geq 0$ that $c_i$ is $0$ or $1$, and that $c_i=0$ for $i\leq \fl{d}$. The base case $c_0=0$ is vacuous. Now suppose as the inductive hypothesis that $c_i$ is $0$ or $1$ and that $c_i=0$ if $i\leq \fl{d}$. Let $x^ay^{d-a}\in S_d$ be a monomial with carry pattern $c$. For ease of notation we set $b=d-a$. Let $\sum_i^M a_ip^i$, $\sum_i^M b_ip^i$, and $\sum_i^M d_ip^i$ be the base-$p$ expansions for $a,b,d$ respectively. Since $c_i+a_i+b_i = d_i+c_{i+1}p$ and $a_i+b_i\leq 2p-2$, using the inductive hypothesis we have that $d_i+c_{i+1}p\leq 2p-1$. Since $d_i\geq 0$ we conclude that $c_{i+1}\leq \lfloor 2-1/p\rfloor =1$, so that $c_{i+1}\in \{0,1\}$. If $i+1\leq \fl{d}$ then $d_i=p-1$ and $c_i=0$ by inductive hypothesis, so $d_i+c_{i+1}p=p-1+c_{i+1}p\leq 2p-2$. Thus, $c_{i+1}p\leq p-1$ and $c_{i+1}=0$, as claimed.

Next, we complete the proof of (1). By Lemma \ref{lemma: two conditions on c} a sequence $c=(c_1,\cdots,c_M)$ belongs to $C(n,2,p)$ if and only if the following holds for all $i$:
$$
0\leq c_i\leq \sum_{k\geq i}d_kp^{k-i}\quad \textnormal{and}\quad 0\leq d_i+pc_{i+1}-c_i\leq 2(p-1).
$$
Since $c_i\in \{0,1\}$ and $d_M\neq 0$ the former condition always holds. The latter condition is equivalent to items (ii) and (iii) in the statement of Proposition \ref{prop: carry only 0 or 1}. Indeed, $d_i+pc_{i+1}-c_i$ is negative if and only if (ii) fails, and it is greater than $2(p-1)$ if and only if (iii) fails.

(3) Suppose $M\neq 0$. Since $d_{\fl{d}}\neq p-1$, items (ii) and (iii) do not prevent $c_{\fl{d}+1}$ from being $1$. If $c_{\fl{d}+1}=\cdots=c_M=1$, then items (ii) and (iii) are satisfied for $i\geq \fl{d}+1$. (4) If $M=0$ then $d\leq p-1$, so no carries can occur.
\end{proof}

\begin{cor}\label{cor:simple2}
The $\GL_2(\k)$-representation $S_d$ is irreducible if and only if one of the following holds:
\begin{enumerate}
    \item $d\leq p-1$,
    \item $d=rp^k-1$ for some $r,k\geq 1$.
\end{enumerate}
\end{cor}

\begin{proof}
If $d\leq p-1$ or $d=rp^k-1$ for some $r,k\geq 1$, then by Proposition \ref{prop: carry only 0 or 1}(2)(4) we have that $C(d,2,p)=\{(0,\cdots,0)\}$, so $S_d$ is irreducible by Theorem \ref{carry-submodule correspondence}. Conversely, if (1) and (2) fail,  then $\fl{d}< M$, so by Proposition \ref{prop: carry only 0 or 1}(3) we have $C(d,2,p)\neq\{(0,\cdots,0)\}$.
\end{proof}

In particular, if $d\leq p-1$ or $d=rp^k-1$ then the only invariant ideal generated in degree $d$ is $\langle x_1,\cdots, x_n\rangle^{d}$.

Next, we investigate generators of carry ideals in two variables. Our main result here is Theorem \ref{thm:genstwovars} below. First, we introduce some notation. Given a carry pattern $c\in C(d,2,p)$ we define
\begin{equation}\label{typecseg}
    \mathcal{Z}(c,d)=\{0\leq k\leq M \mid c_k=0,\;(c_{k-1},d_{k-1})\neq (0,p-1)\}.
\end{equation}
We let $\mathcal{Z}(c,d)=\{0=t_0<t_1<\cdots<t_{\ell}\leq M\}$ 
and make the convention that $t_{\ell+1}=M+1$.
Given an integer $0\leq a\leq d$ with base-$p$ expansion $a=(a_0,\cdots,a_M)$ we define the $c$-\defi{segment of $a$ starting at} $t_r\in \mathcal{Z}(c,d)$ to be
\begin{equation}
a_{[t_r,t_{r+1})}=(a_{t_r},a_{t_r+1},\cdots, a_{t_{r+1}-1}).    
\end{equation}
When $a=d$ we often write $\delta_{t_r}=d_{[t_r,t_{r+1})}$. 
The \defi{content} of a $c$-segment of $a$ is
\begin{equation}\label{defcontent}
\Cont{a_{[t_r,t_{r+1})}} = a_{t_r}p^0+a_{t_r+1}p^1+\cdots + a_{t_{r+1}-1}p^{t_{r+1}-t_r-1}. 
\end{equation}

In other words content of a $c$-segment is the value of the number with base-$p$ expansion given by the segment.
We refer to the set of $c$-segments $\{a_{[t_0,t_{1})}, a_{[t_1,t_{2})},\cdots ,a_{[t_{\ell},t_{M+1})}\}$ as the \defi{type $c$ segmentation of} $a$.

\begin{ex}\label{ex:Segmentationp5d62102}
    Let $p=5$ and $d=62102$ so that the base-$p$ expansion of $d$ is
    $(2,0,4,1,4,4,3)$.
    Let $c = (1,1,0,1,0,0) \in C(d,2,p)$.
    Then $\mathcal{Z}(c,d)=\{0,3,5\}$, and the $c$-segments of $d$ are
         $$
         \delta_0=(2,0,4),\quad \delta_3=(1,4),\quad \delta_5=(4,3).
         $$
         The content of each $c$-segment is
         $$
         \Cont{\delta_0}=2+4\cdot 25=102,\quad \Cont{\delta_3}=1+4\cdot 5=21,\quad \Cont{\delta_5}=4+3\cdot 5=19.
         $$
\end{ex}

We are now ready to state our main result on generators of carry ideals in two variables.

\begin{theorem}\label{thm:genstwovars} Let $c\in C(d,2,p)$ and let $\{\delta_{t_0},\dots,\delta_{t_\ell}\}$ be the type $c$ segmentation of $d$. We have
$$
I_{c,d}=\big(\m^{\cont{\delta_{t_0}}} \big)^{[p^{t_0}]}\big(\m^{\cont{\delta_{t_1}}} \big)^{[p^{t_1}]}\cdots \big(\m^{\cont{\delta_{t_{\ell}}}} \big)^{[p^{t_{\ell}}]},
$$
where $\m=\langle x,y\rangle$ and $(-)^{[p^e]}$ denotes the $p^e$-th Frobenius power. In particular, the generators of $I_{c,d}$ are $x^ay^{d-a}$ where 
\begin{equation}\label{eqn:powertwovars}
a=w_0p^{t_0}+\cdots +w_{\ell}p^{t_{\ell}},\quad 0\leq w_r\leq \cont{\delta_{t_r}}.
\end{equation}
\end{theorem}

We note that, in the statement of Theorem \ref{thm:genstwovars}, we have that $w_r$ is the content of $a_{[t_r,t_{r+1})}$.

\begin{ex}\label{ex:Segmentationp5d62102IdealFactoring}
    Consider $p,c,d$ as in Example~\ref{ex:Segmentationp5d62102}. 
    By Theorem \ref{thm:genstwovars} we have
    \[ I_{c,d} = \m^{102} (\m^{21})^{[5^3]} (\m^{19})^{[5^5]}. \]
\end{ex}

\begin{remark}
Theorem \ref{thm:genstwovars} resembles the Steinberg tensor formula \cite[Corollary II.3.17]{jantzen}, which says that if $a\geq 0$ with base-$p$ expansion $a=(a_0,\cdots,a_M)$, then there is an isomorphism of $\GL_2(\k)$-representations:
$$
L(a)=L(a_0)\otimes L(a_1)^{[p]}\otimes L(a_2)^{[p^2]}\otimes \cdots \otimes L(a_M)^{[p^M]},
$$
where $L(a_e)^{[p^e]}$ denotes the $p^e$-th Frobenius power of $L(a_e)$ \cite[II.3.16]{jantzen}.

For instance, when $c=(0,\cdots,0)$, the ideal $I_{c,d}$ is generated by $L(d)\subseteq S_d$. Since $S_j=L(j)$ for $j\leq p-1$ (see Example \ref{exsmalld}), it follows from the Steinberg tensor formula that
$$
I_{(0),d}=\big(\m^{d_0} \big)\big(\m^{d_1} \big)^{[p]}\big(\m^{d_2} \big)^{[p^2]}\cdots \big(\m^{d_M} \big)^{[p^M]},
$$
which is the same decomposition as in Theorem \ref{thm:genstwovars} if and only if $\mathcal{Z}(c,d)=\{0,\cdots,M\}$ (i.e. $d_i\neq p-1$ for all $i<M$). It is unclear to us the relation between Theorem \ref{thm:genstwovars} and the Steinberg formula when $c\neq (0,\cdots,0)$.
\end{remark}

    We proceed to prove Theorem \ref{thm:genstwovars}. Let $d \geq 0$ and let $c\in C(d,2,p)$. Given the type $c$ segmentation $\{\delta_{t_0},\dots,\delta_{t_\ell}\}$ of $d=(d_0,\dots,d_M)$, we define the set 
    \begin{equation}
     A^c = \{ a = (a_0, \dots,a_M) \mid 0\leq a\leq d \text{ and } \Cont{a_{[t_r,t_{r+1})}} \le \Cont{\delta_{t_r}} \text{ for all $0\leq r\leq \ell$}\},   
    \end{equation} where $\{a_{[t_0,t_{1})}, a_{[t_1,t_{2})},\cdots ,a_{[t_{\ell},t_{M+1})}\}$ is the type $c$ segmentation of the sequence $a$.

\begin{lemma}\label{lem: set Ac}
    Let $c\in C(d,2,p)$ and let $x^ay^{d-a}\in S_d$ have carry pattern $c^\prime$. Then $c^\prime\leq c$ if and only if $a\in A^c$.  
\end{lemma}

\begin{proof}
Let $a=(a_0,\dots,a_M)$ be the base-$p$ expansion of $a$ with type $c$ segmentation $\{\alpha_{t_0},\dots,\alpha_{t_\ell}\}$.
    First, we assume that $c^\prime\leq c$. Then for all $0\leq r\leq \ell$ we have $c_{t_r} = c_{t_{(r+1)}}= 0$, and since $c^\prime\leq c$, we conclude also $c_{t_r}^\prime = c_{t_{(r+1)}}^\prime= 0$. By the defining equations (\ref{carry pattern}) of carry patterns, we have
    $$
        \sum_{k=0}^{t_r-1}(a_k+(d-a)_k)p^k - \sum_{k=0}^{t_r-1}d_kp^k = c_{t_r}^\prime =0,\quad \sum_{k=0}^{t_{(r+1)}-1}(a_k+(d-a)_k)p^k - \sum_{k=0}^{t_{(r+1)}-1}d_kp^k = c_{t_{(r+1)}}^\prime =0.
    $$
   Taking the difference of these equations, we obtain 
    \[\sum_{k=t_r}^{t_{(r+1)}-1}(a_k+(d-a)_k)p^k - \sum_{k=t_r}^{t_{(r+1)}-1}d_kp^k = 0,\]
    from which we deduce the inequality \[\cont{\alpha_{t_r}}p^{t_r} = \sum_{k=t_r}^{t_{(r+1)}-1}a_kp^k \leq \sum_{k=t_r}^{t_{(r+1)}-1}d_kp^k = \cont{\delta_{t_r}}p^{t_r}.\] Therefore, $\cont{\alpha_{t_r}}\leq \cont{\delta_{t_r}}$ for all $0\leq r\leq \ell$, so $a\in A^c$.

    We prove the converse by the contrapositive. Assume that some monomial $x^{a'}y^{d-a'}\in S_d$ has carry pattern $c'$ that is greater than or incomparable to $c$. We will show that $a'=(a_0',\cdots,a_M')$ does not belong to $A^c$. The assumption on $c'$ implies that there exists some $1\leq m\leq M$ for which $c_m'=1$ and $c_m=0$. Let $m$ be the minimal index satisfying this property. We claim that $(c_{m-1},d_{m-1})\neq(0,p-1)$ so that $m=t_r\in \mathcal{Z}(c,d)$ for some $1\leq r\leq \ell$. Indeed, if $(c_{m-1},d_{m-1})=(0,p-1)$, then by minimality of $m$ we have that $c_{m-1}'=0$. Thus, $c_{m-1}'=0$, $c_m'=1$, and $d_{m-1}=p-1$, violating part (1) of Proposition \ref{prop: carry only 0 or 1}. 
    
    Let $t_r=m\in\mathcal{Z}(c,d)$, noting that $t_r\neq t_0=0$, since $c_0=c_0'=0$. Thus, $r\geq 1$. Since $t_r$ is the first index for which $c^\prime_{t_r}>c_{t_r}$, we have $c^\prime_{t_{(r-1)}}\leq c_{t_{(r-1)}}=0$, so that $c^\prime_{t_{(r-1)}}=0$.

    Expanding $c^\prime_{t_{(r-1)}}$ and $c'_{t_r}$ using the defining equations (\ref{carry pattern}) of carry patterns, we have
$$
\sum_{k=0}^{t_{(r-1)}-1}(a'_k + (d-a')_k)p^k - \sum_{k=0}^{t_{(r-1)}-1}d_kp^k = c'_{t_{(r-1)}}p^{t_{(r-1)}} = 0,\quad \sum_{k=0}^{t_r-1}(a'_k + (d-a')_k)p^k - \sum_{k=0}^{t_r-1}d_kp^k = c'_{t_r}p^{t_r} = p^{t_r}.
$$
Taking their difference, we obtain
 \[\sum_{k=t_{(r-1)}}^{t_r-1}a'_k p^k = \sum_{k=t_{(r-1)}}^{t_r-1}d_kp^k  + p^{t_r} - \sum_{k=t_{(r-1)}}^{t_r-1}(d-a')_kp^k.\]
It follows that 
\[
\sum_{k=t_{(r-1)}}^{t_r-1}a'_k p^k > \sum_{k=t_{(r-1)}}^{t_r-1}d_kp^k
.\]
Dividing both sides by $p^{t_{(r-1)}}$, we see that $\cont{a'_{[t_{r-1},t_{r})}} > \cont{\delta_{t_{r-1}}}$, so $a'\notin A^c$.
\end{proof}

\begin{proof}[Proof of Theorem \ref{thm:genstwovars}]
Observe that $a$ is of the form (\ref{eqn:powertwovars}) if and only if $a\in A^c$. By Lemma \ref{lem: set Ac} we conclude that $x^ay^{d-a}$ is a generator of $I_{c,d}$ if and only if $a$ is of the form (\ref{eqn:powertwovars}).

The generators of $\prod_{r=0}^{\ell} \left(\m^{\cont{\delta_{t_r}}} \right)^{[p^{t_r}]}$ are
$$
(x^{w_0p^{t_0}}y^{(\cont{\delta_{t_0}}-w_0)p^{t_0}})\cdot (x^{w_1p^{t_1}}y^{(\cont{\delta_{t_1}}-w_1)p^{t_1}})\cdots (x^{w_{\ell}p^{t_{\ell}}}y^{(\cont{\delta_{t_{\ell}}}-w_{\ell})p^{t_{\ell}}}),\quad 0\leq w_r\leq \cont{\delta_{t_r}},
$$
which have degree $d$, and the following powers of $x$:
$$
w_0p^{t_0}+\cdots +w_{\ell}p^{t_{\ell}},\quad 0\leq w_r\leq \cont{\delta_{t_r}}.
$$
In particular, this ideal has the same generators as $I_{c,d}$.
\end{proof}

In light of Theorem~\ref{thm:genstwovars}, we pose the following question:

\begin{question}
Let $n \geq 3$. Which invariant ideals can be expressed as a product of Frobenius powers of powers of the homogeneous maximal ideal?    
\end{question}
Note that products of Frobenius powers of powers of the homogeneous maximal ideal are invariant.

\section{Free resolutions}\label{freeres}

In this section we begin to investigate graded free resolutions of invariant ideals. Our main focus here is minimal free resolutions, which do not necessarily admit an equivariant structure in positive characteristic. For a discussion of equivariance, see Section \ref{equivariantres}.

\subsection{Generalities on minimal free resolutions of invariant ideals}\label{generalfreeres} Let $S=\k[x_1,\cdots,x_n]$ endowed with its standard grading with $\deg(x_i)=1$ for $i=1,\cdots,n$. Given an invariant ideal $I$, we consider the minimal graded free resolution $\mathbf{F}_{\bullet}$ of $S/I$:
$$
\mathbf{F}_{\bullet}:\quad F_0\xleftarrow{\partial_1} F_1 \overset{\partial_2}\longleftarrow F_2 \overset{\partial_3}\longleftarrow \cdots \overset{\partial_{n-1}}\longleftarrow F_{n-1}\overset{\partial_n}\longleftarrow F_n,
$$
which is an acyclic complex with $\operatorname{coker}(\partial_1)=S/I$ and $\partial_i(F_i)\subseteq \langle x_1,\cdots,x_n\rangle F_{i-1}$ for all $i\geq 1$. Here, the differentials are homogeneous of degree zero and the terms are graded free modules:
$$
F_i=\bigoplus_j S(-j)^{\oplus \beta_{i,j}}.
$$
The multiplicities $\beta_{i,j}$ are the graded Betti numbers. In other words, $\beta_{i,j}=\dim_{\k} \operatorname{Tor}_i^S(S/I,\k)_j$ and $F_i=S\otimes_{\k} \operatorname{Tor}_i^S(S/I,\k)$. Note that $F_0=S$, so that $\beta_{0,j}\neq 0$ if and only if $j=0$, and $\beta_{0,0}=1$. 

The projective dimension of $S/I$, written $\pd_S(S/I)$, is the maximal index $0\leq j \leq n$ for which $F_j\neq 0$. We show that $\pd_S(S/I)$ is the largest valued permitted by Hilbert's Syzygy Theorem \cite[Theorem 1.1]{gos}.

\begin{prop}\label{prop:pd}
    If $I\subseteq S$ is an invariant ideal, then $\pd_S(S/I)= n$.
\end{prop}

\begin{proof}
By Lemma \ref{lemma:origin}, we have that $I_{d } = S_{d }$ for all $d\gg 0$. Let $d $ be minimal for which $I_{d } = S_{d }$. Then every element of $(S/I)_{d -1}$ is annihilated by all elements of the homogeneous maximal ideal. Therefore, $\depth(S/I)=0$, so by the Auslander--Buchsbaum formula \cite[Theorem 1.2.7]{Weyman}, we have $\pd_S(S/I)= n$. 
\end{proof}

Proposition \ref{prop:pd} asserts that we if place the graded Betti numbers $\beta_{i,j}$ into the Betti table (pictured below), then the $n$-th column is the final nonzero column.

    \[\begin{array}{l|cccccc}
        & 0 & 1  & \cdots & i & \cdots & n \\ \hline
       0 & \beta_{0,0} & \beta_{1,1}  & \cdots & \beta_{i,i} & \cdots & \beta_{n,n} \\
       1 & \beta_{0,1} & \beta_{1,2}  & \cdots & \beta_{i,i+1} & \cdots & \beta_{n,n+1} \\
       2 & \beta_{0.2} & \beta_{1,3}  & \cdots & \beta_{i,i+2} & \cdots & \beta_{n,n+2} \\
       \vdots & \vdots & \vdots  & \vdots & \vdots & \vdots & \vdots  \\
       j & \beta_{0.j} & \beta_{1,1+j}  & \cdots & \beta_{i,i+j} & \cdots & \beta_{n,n+j}\\
       \vdots & \vdots  & \vdots & \vdots & \vdots & \vdots & \vdots 
    \end{array}
\]

\medskip

The index of the final nonzero row is the Castelnuovo--Mumford regularity:
$$
\operatorname{reg}(S/I)=\max\{j\mid \beta_{i,i+j}\neq 0\textnormal{ for some $1\leq i\leq n$}\}.
$$
Since $S/I$ has finite length (Lemma \ref{lemma:origin}), it follows from \cite[Corollary 4.4]{gos} that
\begin{prop}\label{regprop}
If $I\subseteq S$ is an invariant ideal, then $\operatorname{reg}(S/I)=\max\{d \mid (S/I)_d\neq 0\}$.    
\end{prop}

In particular, the regularity of $S/I$ is the maximal degree of a monomial not appearing in $I$. We illustrate here with an example.

\begin{ex}\label{ex:betti}
Let $p=3$ and $n=3$, and consider 
$$
I=I_{(0),4}=\langle x^4, x^3y, x^3z, xy^3, xz^3, y^4, y^3z, yz^3, z^4\rangle\subseteq \k[x,y,z].
$$
We see that $(S/I)_6=\k\langle x^2y^2z^2\rangle$ and $(S/I)_7=0$. Therefore, $\operatorname{reg}(S/I)=6$. According to Macaulay2 \cite{GS} the Betti table of $S/I$ is:
$$
\begin{matrix}
 & 0 & 1 & 2 & 3\\
\text{total:} & 1 & 9 & 12 & 4\\
0: & 1 & . & . & .\\
1: & . & . & . & .\\
2: & . & . & . & .\\
3: & . & 9 & 9 & 3\\
4: & . & . & 3 & .\\
5: & . & . & . & .\\
6: & . & . & . & 1
\end{matrix}
$$
We study the representation structure of $\operatorname{Tor}_i^S(S/I,\k)_j$ in Example \ref{exTorcancel} below. In particular, we provide a representation-theoretic reason why the Betti table of $S/I$ does not satisfy the maximal cancellation property.
\end{ex}

It would be interesting to find a general combinatorial formula for the regularity of $S/I$ in terms of the carry patterns corresponding to the generators of $I$. 

\begin{question}
    Is there a formula for Castelnuovo--Mumford regularity of a carry ideal in terms of carry pattern $c$, degree $d$, and number of variables $n$?
\end{question}

In the two variable case we provide an answer (Corollary~\ref{regn2}) to this question in terms of the type $c$ segmentation of $d$.
We remark that, since $\GL_n(\k)$-invariant ideals are $\mathfrak{S}_n$-invariant monomial ideals, one may calculate the regularity using \cite[Theorem on Regularity and Projective Dimension]{raicu}.

\subsection{Representation structure of Tor}\label{sec:reptor} The quotient of $S$ by the homogeneous maximal ideal $\m=\langle x_1,\cdots,x_n\rangle$ is resolved by the Koszul complex $\mathcal{K}_{\bullet}$ (see \cite[Section 17.2]{Eisenbud}): 

\smallskip 

\begin{center}
$\mathcal{K}_{\bullet}:\quad S  \leftarrow S\otimes \k^n\leftarrow S\otimes \bigwedge^2(\k^n)\leftarrow \cdots \leftarrow S\otimes \bigwedge^{n-1}(\k^n)\leftarrow S\otimes \bigwedge^n(\k^n)\leftarrow 0. $\\
\end{center}
Given an invariant ideal $I\subseteq S$, the modules $\operatorname{Tor}_i^S(S/I,\k)$ are endowed with structures as representations of $\GL_n(\k)$ via the Koszul complex $\mathcal{K}_{\bullet}$. Indeed, since $\mathcal{K}_{\bullet}$ is a complex of representations, and $\operatorname{Tor}_i^S(S/I,\k)=H_i(\mathcal{K}_{\bullet}\otimes_S S/I)$, we have that $\operatorname{Tor}_i^S(S/I,\k)$ is the homology of a complex of representations.

\begin{prob}
    Describe the $\GL_n(\k)$-representation structure of the modules $\operatorname{Tor}_i^S(S/I, \k)$, where $I$ is an invariant ideal.
\end{prob}

In Corollary \ref{repontorn2} we provide a formula for $\operatorname{Tor}_i^S(\k[x,y]/I, \k)$ in the Grothendieck group of representations of $\GL_2(\k)$. We remark that the $\mathfrak{S}_n$-module structure of these spaces has been studied in detail for $\mathfrak{S}_n$-invariant monomial ideals in \cite{murai}.

Taking graded pieces of the complex $\mathcal{K}_{\bullet}\otimes_S S/I$, we obtain the following (c.f. \cite[Proposition 2.7]{gos}).

\begin{prop}\label{prop:tor}
Let $I\subseteq S$ be an invariant ideal and let $j\geq 0$. For $1\leq i\leq n$ the $\GL_n(\k)$-representation $\operatorname{Tor}_i^S(S/I,\k)_j$ is isomorphic to the homology, at the term $(S/I)_{j-i}\otimes \bigwedge^i(\k^n)$, of the complex

\smallskip 

\begin{center}
$0\to (S/I)_{j-n}\otimes \bigwedge^n(\k^n)\to (S/I)_{j-n+1}\otimes \bigwedge^{n-1}(\k^n)\to \cdots \to (S/I)_{j-1}\otimes \bigwedge^1(\k^n)\to (S/I)_{j}\otimes \k\to 0. $  \\
\end{center}
\end{prop}

Let $r=\operatorname{reg}(S/I)$. It follows from Proposition \ref{regprop} that we have the following concrete description of $\operatorname{Tor}_n^S(S/I,\k)_r$ (which corresponds to the bottom-right Betti number)

\begin{cor}
Using notation as above, we have

\smallskip

\begin{center}
$\operatorname{Tor}_n^S(S/I,\k)_r=(S/I)_r\otimes \bigwedge^n(\k^n).$\\
\end{center}
\end{cor}

\begin{ex}\label{exTorcancel}
We examine some of the Tor modules from Example \ref{ex:betti}. As $(S/I)_6=\k\langle x^2y^2z^2\rangle$ we have 
$$
\operatorname{Tor}_3^S(S/I,\k)_9=L(2,2,2)\otimes L(1,1,1)=L(3,3,3).
$$
We consider the Tor modules corresponding to $\beta_{3,6}=\beta_{2,6}=3$, which arise as the homology of the complex

\smallskip 

\begin{center}
$0\to (S/I)_{3}\otimes \bigwedge^3(\k^3)\to (S/I)_{4}\otimes \bigwedge^{2}(\k^3)\to  (S/I)_{5}\otimes \k^3 \to (S/I)_{6}\to 0. $  \\
\end{center}
Since $(S/I)_{3}=S_3$ has composition factors $L(3,0,0)$, $L(2,1,0)$, each with multiplicity one, it follows that $(S/I)_{3}\otimes \bigwedge^3(\k^3)$ has composition factors $L(4,1,1)$ and $L(3,2,1)$, of dimensions $3$ and $7$ respectively. Since $\beta_{3,6}=3$, we conclude that 
$$
\operatorname{Tor}_3^S(S/I,\k)_6=L(4,1,1).
$$
We show that $L(4,1,1)$ is not a simple factor of $(S/I)_{4}\otimes \bigwedge^{2}(\k^3)$, and hence $\operatorname{Tor}_2^S(S/I,\k)_6$ is not isomorphic to $L(4,1,1)$. This explains why the Betti table of $S/I_{c,d}$ does not satisfy the maximal cancellation property. It suffices to show that $(S/I)_{4}\otimes \bigwedge^{2}(\k^3)$ has no weight vector of weight $(4,1,1)$. Since $\bigwedge^{2}(\k^3)$ is three dimensional with weights $(1,1,0)$, $(1,0,1)$, and $(0,1,1)$, it suffices to show that $(S/I)_{4}$ has no weight vector of weight $\lambda$ with $\lambda_1=3$. Since $(S/I)_{4}$ is spanned by permutations of $x^2y^2$ and $x^2yz$, the result follows. 
\end{ex}

\subsection{Minimal free resolutions in the case of two variables}
\label{section: Two variable case}
Let $S = \k[x,y]$ with $\Char k = p$ and let $I\subseteq S$ be an invariant ideal. By Proposition \ref{prop:pd} we know that $\pd_S(S/I)=2$.

\begin{prop}\label{theorem: hilbert burch}
    Let $I\subseteq S$ be a monomial ideal with $\depth(S/I)=0$, and order the minimal generators $x^{a_1}y^{b_1},\dots,x^{a_r}y^{b_r}$ of $I$ so that $a_j\geq a_{j+1}$ and $b_j\leq b_{j+1}$ for all $1\leq j
    \leq r$. The minimal free resolution of $S/I$ is
    \[0\leftarrow S/I\leftarrow S \xleftarrow{\begin{bmatrix} x^{a_1}y^{b_1} & \cdots & x^{a_r}y^{b_r}\end{bmatrix}}S^r\xleftarrow{\begin{bmatrix}
    y^{b_2-b_1}\\
    -x^{a_1-a_2} & y^{b_3-b_2}\\
    & -x^{a_2-a_3} & \ddots\\
    & & \ddots & y^{b_r-b_{r-1}}\\
    & & & -x^{a_{r-1}-a_r}
\end{bmatrix}}S^{r-1}\leftarrow 0,\]
 where blank entries of the $r\times (r-1)$ matrix are zero.  
\end{prop}

\begin{proof}
A direct calculation shows that the two matrices above compose to zero, and that the maximal minors of the $r\times (r-1)$ matrix are the minimal generators of $I$ up to a sign. Since $\depth(S/I)=0$, the Hilbert--Burch Theorem \cite[Theorem 20.15b]{Eisenbud} yields that the above complex is a minimal free resolution of $S/I$.
\end{proof}

Using notation as in Section \ref{generalfreeres}, we say that $S/I$ has \defi{pure resolution} if for all $1\leq i\leq n$ there exists a unique $j\geq 0$ for which $\beta_{i,j}\neq 0$. In other words, each column of the Betti table has at most one nonzero entry. An easy necessary condition for $S/I$ to have a pure resolution is that $I$ must be generated in a single degree (so column one of $\beta(S/I)$ has one nonzero entry). In the case of two variables, we provide a complete characterization of invariant ideals with pure resolution.

\begin{prop}
    Let $I$ be an invariant ideal in $S = \k[x,y]$. Then $S/I$ has pure resolution if and only if $I=(\langle x, y \rangle^m)^{[p^e]}$ for some $m\geq 1$ and $e\geq 0$.
\end{prop}

\begin{proof}
    If $I=(\langle x, y \rangle^m)^{[p^e]}$ then $S/I$ has pure resolution by Proposition \ref{theorem: hilbert burch}.

    Conversely, suppose $S/I$ has pure resolution, so that $I$ is generated in a single degree. Then by Proposition \ref{theorem: hilbert burch} we have $I= \langle x^d,x^{d-r}y^r,x^{d-2r}y^{2r},\cdots ,y^d\rangle$ for some $d\geq 1$ and some $1\leq r\leq d$. In particular, $r$ must divide $d$. We will show that $r=p^e$ for some $e\geq 0$. In which case, setting $m=d/r$ will give $I=(\langle x, y \rangle^m)^{[p^e]}$.

    Let $p^s$ be the smallest power of $p$ appearing in the base-$p$ expansion of $d$ (so $p^s$ is the largest power of $p$ dividing $d$). Then the monomial $x^{d-p^s}y^{p^s}\in S_d$ has carry pattern $(0,\cdots,0)$, so it belongs to $I$ by Theorem \ref{carry-submodule correspondence}. In particular, $r$ divides $p^s$, so $r=p^e$ for some $0\leq e\leq s$, as required.
\end{proof}

Now that we understand purity in the two variable case, we pose the following question:
\begin{question}\label{q:Purity}
    Let $n \geq 3$.
    When is the minimal free resolution of $S/I$ pure?
\end{question}
We suspect that $S/I$ has pure resolution if and only if $I$ is a Frobenius power of a power of the homogeneous maximal ideal.

Next, we turn to graded Betti numbers of carry ideals in two variables (see Theorem \ref{thm:genstwovars} for the generators). We continue to write $C(d,2,p)$ for the set of carry patterns of monomials of degree $d$ in two variables. Recall that if $c\in C(d,2,p)$ then $c_i\in \{0,1\}$ for all $1\leq i\leq M$ (see Proposition \ref{prop: carry only 0 or 1}). 

Fix $d\geq 0$ and $c\in C(d,2,p)$. We produce a closed formula for the Betti numbers of $I_{c,d}$ in terms of contents of type $c$ segments of $d$ (see Section \ref{gens2} for notation). We recall the definition of $\mathcal{Z}(c,d)$ as in (\ref{typecseg}):
$$
\mathcal{Z}(c,d)=\{0\leq k\leq M \mid c_k=0,\;(c_{k-1},d_{k-1})\neq (0,p-1)\}.
$$
We let $\mathcal{Z}(c,d)=\{0=t_0<t_1<\dots<t_\ell\leq M\}$ as in (\ref{typecseg}), with the convention that $t_{\ell+1}=M+1$. We write $\{\delta_{t_0},\cdots,\delta_{t_{\ell}}\}$ for the type $c$ segmentation of $d=(d_0,\cdots,d_M)$. For $0\leq r\leq \ell$ we define
\begin{equation}
 \phi_r=p^{t_r} - \sum_{k=1}^{r-1}\cont{\delta_{t_k}}p^{t_k},  \end{equation}
where $\cont{\delta_{t_r}}$ denotes the content of $\delta_{t_r}$ (see (\ref{defcontent})). With this notation, our main result here is the following.

\begin{theorem}\label{THEOREM SFS}
Using notation as above, the following is true about the graded Betti numbers $\beta_{i,j}$ of $S/I_{c,d}$.
\begin{enumerate}
    \item $\beta_{1,j}$ is nonzero if and only if $j=d$, and
    $$
\beta_{1,d}=(\cont{\delta_{t_0}}+1)(\cont{\delta_{t_1}}+1)\cdots (\cont{\delta_{t_{\ell}}}+1).$$

    \item $\beta_{2,j}$ is nonzero only if $j=d+\phi_r$ for some $0\leq r\leq \ell$, and 
    $$
    \beta_{2,d+\phi_r}=\cont{\delta_{t_r}}\cdot (\cont{\delta_{t_{r+1}}}+1)(\cont{\delta_{t_{r+2}}}+1)\cdots (\cont{\delta_{t_{\ell}}}+1).
    $$

\end{enumerate}
\end{theorem}
Note that $\beta_{2,d+\phi_r}$ may be $0$, in particular this occurs when $\cont{\delta_{t_r}}$ is $0$.

\begin{ex}
    Let $p=3$, $n=2$, and $d=30$.
    Using Macaulay2 \cite{GS}, we see that the Betti table of $S/I_{(1,0,1),d}$ is
    $$\begin{matrix}
         & 0 & 1 & 2\\
      \text{total:}
         & 1 & 16 & 15\\
      0: & 1 & . & .\\
      1: & . & . & .\\
      \vdots & \vdots & \vdots & \vdots \\
      28: & . & . & .\\
      29: & . & 16 & 12\\
      30: & . & . & .\\
      \vdots & \vdots & \vdots & \vdots \\
      33: & . & . & .\\
      34: & . & . & 3
      \end{matrix}$$

      \smallskip
      
    The base-$3$ expansion of $d=30$ is $(0,1,0,1)$. Thus $\mathcal{Z}(c,d) = \{0,2\}$, and the type $(1,0,1)$ segments of $d$ are
    $$
    \delta_0=(0,1),\quad \delta_2=(0,1),
    $$
    with contents $\cont{\delta_0}=\cont{\delta_2}=3$. By Theorem~\ref{thm:genstwovars} we obtain that $I_{(1,0,1), d} = \m^3 (\m^3)^{[3^2]}$, which has $\beta_{1,d-1} = (3 + 1)(3 + 1) = 16$ generators, as asserted by Theorem \ref{THEOREM SFS}(1). In order to work out the other Betti numbers, we have
    
    \[\phi_0 = 1, \quad \text{and} \quad\phi_1 = 3^2 - 3 = 6,\]
    so that
    \[\beta_{2, 30 + \phi_0} =\beta_{2, 30 + 1} = 3(3+1) = 12 ,\quad \text{and}\quad 
    \beta_{2, 30 + \phi_1} = \beta_{2, 30 + 6} = 3.
    \]
    We explicate a syzygy of each degree. The  following monomials are among the generators of $I_{(1,0,1)}$:
    \[x^{30} ,\quad 
      x^{29}y ,\quad
      x^{27}y^3 ,\quad
      x^{21}y^9 .
    \]
 There is a syzygy of degree $\phi_0=1$ between $x^{30}$ and $x^{29}y$, and there is a syzygy of degree $\phi_1=6$ between $x^{27}y^3$ and $x^{21}y^9$.
\end{ex}

Using Theorem \ref{THEOREM SFS} we can calculate the Betti table of large examples without having the generators.

 \begin{ex}\label{ex: running ex syz}
     Let $p=5$, and let $d=62102$, which has base-$p$ expansion $(2,0,4,1,4,4,3)$. We calculate the Betti number of $S/I_{c,d}$, where $c =(1,1,0,1,0,0)\in C(d,2,p)$. 

     Here, $\mathcal{Z}(c,d)=\{0,3,5\}$, and the $c$-segments of $d$ are
     $$
     \delta_0=(2,0,4),\quad \delta_3=(1,4),\quad \delta_5=(4,3).
     $$
     The content of each $c$-segment is
     $$
     \cont{\delta_0}=2+4\cdot 25=102,\quad \cont{\delta_3}=1+4\cdot 5=21,\quad \cont{\delta_5}=4+3\cdot 5=19.
     $$
     We conclude from Theorem \ref{THEOREM SFS}(1) that
     $$\beta_{1,d}=(102+1)(21+1)(19+1)=45320.$$
  We compute the degrees $\phi_r$ of the syzygies for $0\leq r\leq 2$:
 $$
 \phi_0=5^0=1,\quad \quad \phi_1=5^{3}- \cont{(2,0,4)}5^0 = 23,\quad \quad \phi_2=5^5 -\cont{(2,0,4)}5^0 - \cont{(1,4)}5^3 = 398.
 $$
 The multiplicities of each degree of syzygy are
 $$
 \beta_{2,62103}=102\cdot (21+1)(19+1)=4488,\quad \beta_{2,62125}=21\cdot (19+1)=420,\quad \beta_{2,62500}=19.
 $$
 Therefore, the regularity is $\operatorname{reg}(S/I_{c,d})=62498$.
\end{ex}

\begin{proof}[Proof of Theorem \ref{THEOREM SFS}]
Part (1) is an immediate consequence of Theorem \ref{thm:genstwovars}, as $I_{c,d}$ is generated in degree $d$ and $\beta_{0,d}$ is the number of minimal generators.   

(2) By Proposition \ref{theorem: hilbert burch} we have  $\beta_{2,d+j}$ is nonzero only if there is $0\leq a\leq d-1$ such that $x^{a+j}y^{d-a-j}$ and $x^ay^{d-a}$ are lexicographically consecutive generators of $I_{c,d}$. Furthermore, $\beta_{2,d+j}$ is the quantity of such $a$'s.

Let $x^ay^{d-a}$ be a minimal generator of $I_{c,d}$ with $0\leq a\leq d-1$. By Theorem \ref{thm:genstwovars} we have that
\begin{equation}\label{sumofa}
a=w_0p^{t_0} + w_1p^{t_1}+w_2p^{t_2}+\cdots +w_{\ell}p^{t_{\ell}},
\end{equation}
for some $0\leq w_k\leq \cont{\delta_{t_k}}$. Let $x^{a'}y^{d-a'}$ be the generator of $I_{c,d}$ immediately before $x^ay^{d-a}$ lexicographically, and let $r$ be minimal such that $w_r<\cont{\delta_{t_r}}$. If follows from Theorem \ref{thm:genstwovars} that
$$
a'=0 p^{t_0}+\cdots+0 p^{t_{r-1}}+(w_r+1)p^{t_r}+w_{r+1}p^{t_{r+1}}+\cdots+w_{\ell}p^{t_{\ell}},
$$
so that $a-a'=\phi_r$. Thus, $\beta_{2,d+j}\neq 0$ only if $j=\phi_r$ for some $1\leq r\leq \ell$. Furtheremore, $\beta_{2,d+\phi_r}$ is the quantity of $0\leq a\leq d-1$ as in (\ref{sumofa}) such that $w_k=\cont{\delta_{t_k}}$ for all $1\leq k<r$ and $w_r<\cont{\delta_{t_r}}$. Therefore, 
$$
    \beta_{2,d+\phi_r}=\cont{\delta_{t_r}}\cdot (\cont{\delta_{t_{r+1}}}+1)(\cont{\delta_{t_{r+2}}}+1)\cdots (\cont{\delta_{t_{\ell}}}+1),
    $$
as claimed.
\end{proof}

As an application, we give a formula for the maximal degree of a monomial not in $I_{c,d}$. The following is an immediate consequence of Proposition \ref{regprop} and Theorem \ref{THEOREM SFS}, noting that $\phi_i<\phi_j$ for all $i<j$.

\begin{cor}\label{regn2} Using notation as above
$$
\operatorname{reg}(S/I_{c,d})=\max\{e\geq 0 \mid (S/I_{c,d})_e\neq 0  \} = d+\phi_{\ell}-2.
$$
\end{cor}

Now that we know the Betti numbers of carry ideals, we turn to the problem of determining the structure of $\operatorname{Tor}_i^S(S/I,\k)$ as a representation. Since the Betti table of $S/I_{c,d}$ has at most one nonzero entry on each diagonal, we obtain the following  consequence of Proposition \ref{prop:tor}.

\begin{cor}\label{repontorn2}
Let $I=I_{c,d}\subseteq S=\k[x,y]$ (or any invariant ideal generated in a single degree $d$). For $j\geq 0$ we have the following in the Grothendieck group of representations of $\GL_n(\k)$:

\smallskip

\begin{center}
$\big[\operatorname{Tor}_2^S(S/I,\k)_j\big] = \big[(S/I)_{j-2}\otimes \bigwedge^2(\k^2)\big] - \big[(S/I)_{j-1}\otimes \k^2\big] + \big[(S/I)_j\big].$\\
\end{center}
\end{cor}
    
\begin{ex}\label{exfindingtor}
Let $S=\mathbf{k}[x,y]$ with $\operatorname{char}(\mathbf{k})=2$, and consider the ideal $I=I_{(0,0),5}=\langle x^5, x^4y, xy^4, y^5\rangle$. By Proposition \ref{theorem: hilbert burch} the minimal free resolution of $S/I$ is:
   \[0\leftarrow S/I\leftarrow S \xleftarrow{\begin{bmatrix} x^5 & x^4y & xy^4 & y^5\end{bmatrix}}S^4\xleftarrow{\begin{bmatrix}
    y & 0 & 0\\
    -x & y^3 & y\\
    0 & -x^3 & x
\end{bmatrix}}S^3\leftarrow 0,\]
so the nonzero graded Betti numbers are
$$
\beta_{0,0}=1,\quad\beta_{1,5}=4,\quad \beta_{2,6}=2.\quad \beta_{2,8}=1.
$$
Since $I$ is generated by $L(5,0)\subseteq S_5$ we have $\operatorname{Tor}_1^S(S/I,\k)_5=L(5,0)$. To find $\operatorname{Tor}_2^S(S/I,\k)_6$ and $\operatorname{Tor}_2^S(S/I,\k)_8$ we use Corollary \ref{repontorn2}. Since $(S/I)_4=S_4$, $(S/I)_5=L(3,2)$, and $(S/I)_6=L(3,3)$, we obtain

\smallskip

\begin{center}
$\big[\operatorname{Tor}_2^S(S/I,\k)_6\big] = \big[S_4\otimes \bigwedge^2(\k^2)\big] - \big[L(3,2)\otimes \k^2\big] + \big[L(3,3)\big].$\\
\end{center}
As $\bigwedge^2(\k^2)=L(1,1)$ and $\k^2=L(1,0)$, this expression is equal to
$$
\big(\big[L(5,1)\big]+\big[L(4,2)]+\big[L(3,3)\big]\big)-\big(\big[L(4,2)]+2\big[L(3,3)]\big)+\big(\big[L(3,3)]\big)=\big[L(5,1)\big],
$$
so $\operatorname{Tor}_2^S(S/I,\k)_6=L(5,1)$. Next, we have 

\smallskip

\begin{center}
$\operatorname{Tor}_2^S(S/I,\k)_8 = L(3,3)\otimes \bigwedge^2(\k^2)=L(4,4).$\\
\end{center}
We recover this calculation using an equivariant free resolution in Example \ref{equivfreeex} below.
\end{ex}

\subsection{Remarks on equivariant free resolutions}\label{equivariantres} Let $I\subseteq S$ be an invariant ideal. By an \defi{equivariant free resolution}  of $S/I$ we mean a free resolution $\{(F_i, \partial_i)\}_{i\geq 0}$ of $S/I$ over $S$ such that: (i) each $F_i$ is isomorphic to $S\otimes W_i$, where $W_i$ is a representation of $\GL_n(\k)$, (ii) the differentials $\partial_i$ are equivariant, meaning that $g\cdot\partial_i(m)=\partial_i(g\cdot m)$ for all $g\in \GL_n(\k)$ and all $m\in F_i$. 

Since the representation theory of $\GL_n(\k)$ is not semi-simple in positive characteristic, a minimal free resolution need not admit an equivariant structure, as demonstrated by the following example.

\begin{ex}
 Let $S=\mathbf{k}[x,y]$ with $\operatorname{char}(\mathbf{k})=2$, and consider the ideal $I=\langle x^4,x^3y^2, x^2y^3, y^4\rangle$, which is the sum of $I_{(0,0),4}=\langle x^4, y^4\rangle$ and $\langle x ,y\rangle ^5$. We consider the natural surjection
 $$
I \xleftarrow{\begin{bmatrix} x^4 & x^3y^2 & x^2y^3 & y^4\end{bmatrix}}S^4=:F.
 $$
 We will show that $F$ cannot be endowed with an equivariant structure that makes this morphism equivariant. Let $\phi$ be the map from $F$ to $I$ above, and let $e_1,e_2,e_3,e_4$ be the basis for $F$ such that $\phi(e_1)=x^4$, $\phi(e_2)=x^3y^2$, $\phi(e_3)=x^2y^3$, $\phi(e_4)=y^4$. Suppose for contradiction that $F$ is a representation that makes $\phi$ equivariant. Then $F=S\otimes (U_4\oplus U_5)$, where $U_4$ and $U_5$ are representations with $U_4=\langle e_1,e_4\rangle$ and $U_5=\langle e_2,e_3\rangle$. We have
 $$
 \begin{bmatrix}
    1 & 1\\
    0 & 1
\end{bmatrix}\cdot x^3y^2=x^5+x^3y^2\quad \implies \quad 
 \phi\left(\begin{bmatrix}
    1 & 1\\
    0 & 1
\end{bmatrix}\cdot e_2\right)=x^5+x^3y^2.
 $$
 Thus, $U_5$ contains a vector that gets sent to $x^5$ under $\phi$. This is a contradiction, as such a vector must be a linear combination of $e_2,e_3$, and hence gets sent to a linear combination of $x^3y^2$ and $x^2y^3$.
 \end{ex}

Despite the example above, equivariant resolutions exist, and may be constructed inductively as follows (see \cite[Section 2.3]{broer} for more details). Let $i\geq 0$ and let $m_1,\cdots,m_r$ be minimal generators for $\operatorname{ker}(\partial_i:F_i\to F_{i-1})$ with the convention that $F_{-1}=S/I$. Let $W_{i+1}$ be the $\GL_n(\k)$ subrepresentation of $F_i$ generated by $m_1,\cdots, m_r$, and inductively define $F_{i+1}=S\otimes W_{i+1}$ and $\partial_{i+1}(1\otimes m)=m\in F_{i-1}$. The resulting complex $\{(F_i, \partial_i)\}_{i\geq 0}$ is an equivariant free resolution of $S/I$, but it may not be minimal when $\operatorname{char}(\mathbf{k})=p>0$. The issue is that the representations $W_{i}$ need not be semi-simple, so if $W_{i}$ is generated in multiple degrees, $F_{i}$ may have superfluous generators.

In the next example, we use this algorithm above to construct a (non-minimal) equivariant free resolution.

\begin{ex} \label{equivfreeex}
We find an equivariant free resolution in the situation of Example \ref{exfindingtor}, where $\operatorname{char}(\mathbf{k})=2$ and $I=I_{(0,0),5}=\langle x^5, x^4y, xy^4, y^5\rangle\subseteq \k[x,y]$. Using notation as above, we have $W_1=L(5,0)$ so that $F_1=S\otimes L(5,0)=S^4$. The first syzygies are
$$
v_{(5,1)}=y\otimes x^5-x\otimes x^4y,\quad v_{(4,4)}=y^3\otimes x^4y-x^3\otimes xy^4,\quad v_{(1,5)}=y\otimes xy^4-x\otimes y^5,
$$
where we have indexed them with their torus weights. Here, $v_{(5,1)},v_{(1,5)}\in S_1\otimes L(5,0)$ and $v_{(4,4)}\in S_3\otimes L(5,0)$. When we act on these with the elementary matrices we obtain:
$$
\begin{bmatrix}
    1 & 1\\
    0 & 1
\end{bmatrix}\cdot v_{(5,1)}=v_{(5,1)},\quad \begin{bmatrix}
    1 & 0\\
    1 & 1
\end{bmatrix}\cdot v_{(5,1)}=v_{(5,1)}+v_{(1,5)},
$$
$$
\begin{bmatrix}
    1 & 1\\
    0 & 1
\end{bmatrix}\cdot v_{(1,5)}=v_{(5,1)}+v_{(1,5)},\quad \begin{bmatrix}
    1 & 0\\
    1 & 1
\end{bmatrix}\cdot v_{(1,5)}=v_{(1,5)},
$$
$$
\begin{bmatrix}
    1 & 1\\
    0 & 1
\end{bmatrix}\cdot v_{(4,4)}=v_{(4,4)}+(x^2+xy+y^2)v_{(5,1)},\quad \begin{bmatrix}
    1 & 0\\
    1 & 1
\end{bmatrix}\cdot v_{(4,4)}=v_{(4,4)}+(x^2+xy+y^2)v_{(1,5)}.
$$
In particular, the subrepresentation $W_2\subseteq S\otimes L(5,0)$ generated by $v_{(5,1)}$, $v_{(1,5)}$, and $v_{(4,4)}$ is spanned by the weight vectors 
\begin{equation}\label{eq:weight}
v_{(5,1)},\; v_{(1,5)},\; v_{(4,4)},\; x^2v_{(5,1)}, \;xyv_{(5,1)},\; y^2v_{(5,1)}, \;x^2v_{(1,5)},\; xyv_{(1,5)}, \;y^2v_{(1,5)}.
\end{equation}
The latter six span a representation that is isomorphic to $S_2\otimes L(5,1)$, which is a nontrivial extension of $L(6,2)$ by $L(7,1)$. Putting it all together, we conclude that $W_2$ is length four with composition factors $L(5,1)$, $L(4,4)$, $L(7,1)$, and $L(6,2)$. More precisely, there is an isomorphism $W_2\cong L(5,1)\oplus W_2'$, where $W_2'$ is characterized by the following two non-split short exact sequences:
$$
0\longrightarrow W_2''\longrightarrow W_2'\longrightarrow L(4,4)\longrightarrow 0,\quad 0\longrightarrow L(7,1)\longrightarrow W_2'' \longrightarrow L(6,2)\longrightarrow 0.
$$
Here, $L(5,1)$ is spanned by $v_{(5,1)}$, $v_{(1,5)}$, $L(4,4)$ is spanned by $v_{(4,4)}$,  $L(7,1)$ is spanned by $x^2v_{(5,1)}$, $y^2v_{(5,1)}$, $x^2v_{(1,5)}$, $y^2v_{(1,5)}$, and  $L(6,2)$ is spanned by $xyv_{(5,1)}$, $xyv_{(1,5)}$.

There are six syzygies between the vectors in (\ref{eq:weight}):
$$
u_{(7,1)}=x^2\otimes v_{(5,1)}-1\otimes x^2v_{(5,1)},\quad u_{(6,2)}=xy\otimes v_{(5,1)}-1\otimes xyv_{(5,1)},\quad u_{(5,3)}=y^2\otimes v_{(5,1)}-1\otimes y^2v_{(5,1)},
$$
$$
u_{(3,5)}=x^2\otimes v_{(1,5)}-1\otimes x^2v_{(1,5)},\quad u_{(2,6)}=xy\otimes v_{(1,5)}-1\otimes xyv_{(1,5)},\quad u_{(1,7)}=y^2\otimes v_{(1,5)}-1\otimes y^2v_{(1,5)}.
$$
It is straightforward to verify that the span of these vectors is isomorphic to $S_2\otimes L(5,1)$ (in particular, after acting on these by elementary matrices, we do not get any new weight vectors). Setting $W_3=S_2\otimes L(5,1)$ we obtain the following (non-minimal) equivariant free resolution of $S/I$:
\[ S \xleftarrow{\begin{bmatrix} x^5 & x^4y & xy^4 & y^5\end{bmatrix}}S\otimes W_1\xleftarrow{\begin{bmatrix}
    y & 0 & 0 & x^2y & xy^2 & y^3 & 0 & 0 & 0\\
    -x & y^3 & 0 & -x^3 & -x^2y & -xy^2 & 0 & 0 & 0\\
    0 & -x^3 & y & 0 & 0 & 0 & x^2y & xy^2 & y^3\\
    0 & 0 & -x & 0 & 0 & 0 & -x^3 & -x^2y & -xy^2
\end{bmatrix}}S\otimes W_2\]
\[ S\otimes W_2 \xleftarrow{\begin{bmatrix}
    x^2 & xy & y^2 & 0 & 0 & 0\\
    0 & 0 & 0 & 0 & 0 & 0\\
    0 & 0 & 0 & x^2 & xy & y^2\\
    -1 & 0 & 0 & 0 & 0 & 0\\
    0 & -1 & 0 & 0 & 0 & 0\\
    0 & 0 & -1 & 0 & 0 & 0\\ 0 & 0 & 0 & -1 & 0 & 0\\
    0 & 0 & 0 & 0 & -1 & 0\\
    0 & 0 & 0 & 0 & 0 & -1\end{bmatrix}} S\otimes W_3\xleftarrow{} 0.\]

After tensoring this resolution with $\mathbf{k}\cong S/\langle x,y\rangle$ we see that we have the following description of the Tor modules as representations of $\GL_2(\mathbf{k})$:
$$
\operatorname{Tor}_0^S(S/I,\mathbf{k})\cong \mathbf{k},\quad 
\operatorname{Tor}_1^S(S/I,\mathbf{k})\cong L(5,0),\quad \operatorname{Tor}_2^S(S/I,\mathbf{k})\cong L(5,1)\oplus L(4,4),
$$
recovering the calculations in Example \ref{exfindingtor}.
\end{ex}

\begin{question}
    When does the minimal free resolution of an invariant ideal admit an equivariant structure?
\end{question}
We suspect that the minimal free resolution admits an equivariant structure if and only if it is pure (see Question~\ref{q:Purity}).

\begin{prob}
    Following the algorithm above, calculate equivariant resolutions of carry ideals in two variables, and determine the submodule lattice of the representations $W_i$.
\end{prob}

\section*{Acknowledgments} We thank Claudiu Raicu for suggesting this project, and we thank Zeus Dantas e Moura and Keller VandeBogert for helpful conversations. This project was initiated at the Combinatorics \& Algebra REU in 2023 at the University of Minnesota - Twin Cities, funded by National Science Foundation Grant No. 1745638.

\end{document}